\pdfoutput=1
\documentclass[11pt]{article}

\usepackage[utf8]{inputenc}
\usepackage[margin=1in]{geometry}
\usepackage{graphicx}
\usepackage{xcolor}
\usepackage[authoryear,round]{natbib}
\setcitestyle{citesep={;},aysep={,},yysep={;}}

\usepackage{amsmath,amsfonts,bm}

\def\eqref#1{\ref{#1}}

\def\1{\bm{1}}

\DeclareMathAlphabet{\mathsfit}{\encodingdefault}{\sfdefault}{m}{sl}
\SetMathAlphabet{\mathsfit}{bold}{\encodingdefault}{\sfdefault}{bx}{n}

\newcommand{\E}{\mathbb{E}}

\newcommand{\R}{\mathbb{R}}

\newcommand{\KL}{D_{\mathrm{KL}}}

\DeclareMathOperator*{\argmin}{arg\,min}

\newcommand{\Cc}{\mathcal{C}}
\newcommand{\Wc}{\mathcal{W}}

\newcommand{\Oc}{\mathcal{O}}

\newcommand{\Lc}{\mathcal{L}}
\newcommand{\Nc}{\mathcal{N}}

\newcommand{\Ac}{\mathcal{A}}

\newcommand{\Ex}{\mathbb{E}}

\newcommand{\Pc}{\mathcal{P}}
\newcommand{\Fc}{\mathcal{F}}

\newcommand{\norm}[1]{\left\lVert#1\right\rVert}

\newcommand{\ito}{It\^{o}}

\renewcommand{\exp}[1]{\mathrm{e}^{#1}}

\newcommand{\dd}{\mathrm{d}}

\newcommand{\id}{I_d}
\DeclareMathOperator*{\TV}{TV}

\let\P\relax
\DeclareMathOperator{\P}{\mathbb{P}}
\DeclareMathOperator{\Q}{\mathbb{Q}}
\DeclareMathOperator{\law}{law}

\newcommand{\discX}{\mathbb{X}}
\newcommand{\discV}{\mathbb{V}}
\newcommand{\discZ}{\mathbb{Z}}
\newcommand{\step}{h}
\newcommand{\aux}[1]{{{#1}^{\text{aux}}}}

\newcommand{\pot}{\Psi}

\newcommand{\mean}{m}
\newcommand{\var}{\Sigma}

\newcommand{\inner}[2]{\left\langle#1,#2\right\rangle}

\newcommand{\friction}{\gamma}
\newcommand{\lip}{L}

\newcommand{\eg}{\textit{e.g.}}

\newcommand{\cf}{\textit{cf.}}

\newcommand{\chibound}{\alpha}
\newcommand{\dtbound}{\beta}
\renewcommand{\epsilon}{\varepsilon}
\newcommand{\minimizerbound}{\xi}
\newcommand{\score}{\mathrm{score}}

\newcommand{\dtfric}{\nu}
\usepackage{amsmath, amsthm, amssymb}
\usepackage[colorlinks=true,linkcolor=blue!60!black,citecolor=blue!60!black,urlcolor=blue!60!black]{hyperref}
\usepackage{url}
\usepackage{glossaries-extra}
\usepackage{cleveref}
\usepackage{mathtools}
\newabbreviation{lsi}{LSI}{log-Sobolev inequality}
\newabbreviation{pi}{PI}{Poincar\'e inequality}
\newabbreviation{sde}{SDE}{stochastic differential equation}
\newabbreviation{pde}{PDE}{partial differential equation}
\newabbreviation{ode}{ODE}{ordinary differential equation}
\newabbreviation{fpe}{FPE}{Fokker--Planck equation}
\newabbreviation{mc}{MC}{Markov chain}
\newabbreviation{mcmc}{MCMC}{Markov chain Monte Carlo}
\newabbreviation{ula}{ULA}{unadjusted Langevin algorithm}
\newabbreviation{daz}{DAZ}{diffusion at absolute zero}
\newabbreviation{smc}{SMC}{sequential Monte-Carlo}
\newabbreviation{ou}{OU}{Ornstein--Uhlenbeck}
\newabbreviation{em}{EM}{Euler-Maruyama}
\newabbreviation{angaul}{ANGAUL}{annealed gradient adjusted underdamped Langevin dynamics}
\newabbreviation{gaul}{GAUL}{gradient adjusted underdamped Langevin dynamics}
\newabbreviation{name}{ANGAUL}{annealed gradient adjusted underdamped Langevin dynamics}
\newabbreviation{pogmdm}{POGMDM}{Product of Gaussian Mixture Diffusion Model}
\newabbreviation{gmm}{GMM}{Gaussian mixture model}
\newabbreviation{kl}{KL}{Kullback--Leibler}

\newtheorem{coroll}{Corollary}[section] 
\newtheorem{lemma}{Lemma}[section] 
\newtheorem{theorem}{Theorem}[section] 
\newtheorem{remark}{Remark}[section]
\newtheorem{assumption}{Assumption}[section]
\crefname{assumption}{assumption}{assumptions}
\Crefname{assumption}{Assumption}{Assumptions}

\newtheorem{example}{Example}[section]
\crefname{example}{example}{examples}
\Crefname{example}{Example}{Examples}
\crefname{coroll}{corollary}{corollaries}
\Crefname{coroll}{Corollary}{Corollaries}
\crefname{equation}{}{}

\title{Annealed underdamped Langevin Dynamics}
\author{Laurenz Nagler \and Alexander Falk \and Andreas Habring}
\date{}

\begin{document}

\maketitle

\begin{abstract}
In sampling, different flavours of annealing, tempering, or other successive approximation approaches are widely used and studied. In this work we investigate annealed underdamped Langevin sampling for logconcave distributions of the form $\pi(x) \propto e^{-\Psi(x)}$ for a potential $\Psi:\mathbb{R}^d \rightarrow \mathbb{R}$. That is, we assume access to the score of a general approximating family of distributions $(\pi_{\tau})_{\tau\in [0,1]}$ with $\pi_0=\pi$. This family is used in a time-inhomogeneous underdamped Langevin process with moving target $\pi_{\tau(t)}$ and an appropriate annealing schedule $t\mapsto \tau(t)$ such that $\tau(t)\downarrow 0$ as $t\to T$ for some $T>0$.
While the degeneracy of the Brownian motion prohibits conventional Girsanov arguments, we are able to circumvent this issue by relying on an appropriately designed auxiliary process. To the best of our knowledge these are the first (quantitative) results for sampling using time-inhomogeneous underdamped Langevin dynamics. Moreover, we obtain an iteration complexity of $\mathcal{O}(\epsilon^{-2})$ in total variation distance in comparison to $\mathcal{O}(\epsilon^{-6})$ for annealed overdamped Langevin sampling. While we currently assume convexity, our proof strategies are amenable to generalization for non-logconcave targets and enable for several future research directions.
\end{abstract}

\section{Introduction}
In this article we consider the problem of sampling from a probability distribution \(\pi_x\) on \(\R^d\) which admits a density with respect to the Lebesgue measure \(\lambda\) of the form 
\begin{equation}\label{eq:gibbs_target}
    p_x(x) := \frac{\dd \pi_x}{\dd \lambda}(x) = \frac{\exp{-\pot(x)} }{\int \exp{-\pot(y)} \dd y},
\end{equation}
for a \textit{potential} \(\pot:\R^d \rightarrow \R\) via annealed underdamped Langevin dynamics. That is, for an appropriate family of approximating potentials $(\pot_{\tau})_{\tau\in[0,1]}$ with $\pot_0 = \pot$ and an annealing schedule $\tau :[0,T]\rightarrow[0,1]$ with $\tau(T)=0$ we are interested in the system
\begin{equation}\label{eq:an_under_sde}
\begin{cases}
        \dd X_t &= V_t \, \dd t\\
        \dd V_t &= \{-\friction_{\tau(t)} V_t - \nabla_x \pot_{\tau(t)}(X_t)\} \, \dd t + \sqrt{2\friction_{\tau(t)} } \, \dd W_t
    \end{cases}
    \tag{ANULD}
\end{equation}
where \((W_t)_{t \geq 0} \in \R^d\) is Brownian motion and  $\friction_{\tau(t)}$ a time-dependent friction accounting for the varying potential. 

The simplest, and probably most famous, stochastic process used to draw samples from $\pi_x$ is the \emph{overdamped Langevin diffusion}
\begin{equation}\label{eq:langevin_diffusion}
    \dd X_t = -\nabla \pot(X_t)\dd t + \sqrt{2}\dd W_t.
    \tag{OLD}
\end{equation}
Under Lipschitz continuity and appropriate growth of the drift~\eqref{eq:langevin_diffusion} admits a unique strong solution \citep{Kloeden_Pearson_1977} and is ergodic in various norms~\citep{dalalyan2017theoretical,durmus2017non,durmus2019analysis,durmus2019high,roberts1996exponential}. Despite its popularity, it is known that~\eqref{eq:langevin_diffusion} suffers from poor mixing especially in multimodal settings, which motivated the design and study of methods of accelerating~\eqref{eq:langevin_diffusion}. We broadly categorize acceleration methods in two groups: \emph{Direct acceleration} and \emph{successive approximation}. 

\paragraph{Direct acceleration versus successive approximation}

\emph{Direct acceleration} relies on the use of~\glspl{sde}, different from~\eqref{eq:langevin_diffusion}, which exhibit the target as their stationary measure, but yield faster ergodic convergence. A way to achieve this is by considering lifted variants of \eqref{eq:langevin_diffusion}. Most notably the \textit{underdamped or kinetic\footnote{both terms are interchangeably used in the literature} Langevin dynamics}~
\begin{equation}\label{eq:kinetic_langevin_time_hom}
        \dd V_t = \{-\friction V_t - \nabla_x\pot(X_t)\} \, \dd t + \sqrt{2\friction } \, \dd W_t,\quad \dd X_t = V_t \, \dd t
    \tag{ULD}
\end{equation}
which augments the space with a velocity \((V_t)_{t\geq0}\) and a friction \(\friction > 0\) (\cf~\citep{eberle2026non_rev_lift} for a more general study of non-reversible lifts). 
Indeed, under appropriate assumptions on the drift and the friction, the law of the solution of \cref{eq:kinetic_langevin_time_hom} converges at a ballistic rate\footnote{the term \emph{ballistic} refers to complexity scaling as $\tilde{\Oc}(\sqrt{\kappa})$ where $\kappa$ is the condition of the problem} to $\pi(\dd x,\dd v) \propto \pi_x(\dd x)\exp{-|v|^2/2}\dd v$ in continuous time \citep{Cao_2023,borkowski2026fenchelgameunderdampedlangevin} and a sublinear rate in discrete time \citep{altschuler2026shiftedcompositionivballistic,borkowski2026fenchelgameunderdampedlangevin}, thus, translating the speedup of Nesterov acceleration~\citep{nesterov1983nag} from optimization to sampling.
Moreover, it was recently shown that \eqref{eq:kinetic_langevin_time_hom} is an optimal lift of \eqref{eq:langevin_diffusion} with respect to relaxation time to the equilibrium \citep{eberle2026non_rev_lift}. While the acceleration of underdamped Langevin for convex targets is, thus, already rather well-understood, its benefits for mode exploration for multimodal targets are limited.

\emph{Successive approximation} instead considers a family of distributions $(\pi(\tau))_{\tau\in[0,1]}$ such that $\pi(0) = \pi$ and for $\tau>0$, $\pi(\tau)$ is increasingly well-behaved. During sampling, we follow the family from $\tau=1$ to $\tau=0$ and the fact that $\pi(\tau)$ is initially well-behaved enables accurate initialization at $\pi(1)$ and helps to explore all modes of the target quickly and thus aids mixing~\citep{papamakarios2021normalizing,song2019generative,song2021scorebased,habring2026diffusion,habring2026forward,chehab2025provable_tempering,albergo2025stochastic}. Within the literature on successive approximation there are once again two distinct approaches, which can be referred to as \emph{traversing} versus \emph{trailing} $\pi(\tau)$. Diffusion models and flows belong to the first category. That is, when simulating a reverse diffusion/flow, up to numerical errors, the generated process $(X_\tau)_\tau$ satisfies precisely $X_\tau \sim \pi(\tau)$~\citep{song2021scorebased}. In this work, we let the dynamics \emph{trail} $\pi(\tau)$. Concretely, the annealed
dynamics~\eqref{eq:an_under_sde} are underdamped Langevin dynamics driven by a
\emph{moving target} $\pi(\tau(t)) \propto \exp{-\pot_{\tau(t)}}$: at time $t$, the process targets $\pi(\tau(t))$ equilibrating as \(t \rightarrow T\). 
What is gained in return is that the \emph{early potentials} are better behaved, which facilitates the above mentioned mode
exploration and improves convergence~\citep{habring2026forward}. For overdamped
Langevin dynamics, this annealing paradigm is by now well
understood~\citep{Cattiaux2025DiffusionAL, cordero2025non, guo2025provable,
chehab2025provable_tempering, habring2026forward} and is, in fact, a direct precursor of diffusion models~\citep{song2019generative, song2021scorebased}. Its underdamped counterpart, however, has not yet been theoretically analyzed except for simulated annealing, that is, in the context of non-convex optimization~\cite{monmarche2018hypocoercivity,journel2022convergence}. We attribute this to
hypocoercivity: since the Brownian motion acts only on the velocity, the standard
arguments for the overdamped setting are not applicable. While hypocoercivity has at this point been handled elegantly in the time-homogeneous setting (see \cref{sec:related_work}), it is not well researched in the time-inhomogeneous, that is, annealed setting. The goal of the present manuscript is to provide a quantitative convergence analysis in this case.

\paragraph{A case for~\eqref{eq:an_under_sde}}
The use of underdamped rather than overdamped dynamics for annealing is motivated both by the extensive literature on the accelerating properties of kinetic Langevin dynamics and by empirical evidence~\citep{zilberstein2023annealedunderdamped}: If underdamped Langevin converges faster than its overdamped counterpart, one may expect that the same holds in the annealed setting.

While most contemporary research on successive approximation focuses on
\emph{traversing} \(\pi(\tau)\), as in diffusion and flow models, there are
several arguments in favour of \emph{trailing} it, as in annealing: Conceptually, annealing is rather flexible with regards to the used path $\pi(\tau)$. Essentially as long as $\tau\mapsto \pi(\tau)$ is absolutely continuous, the annealed dynamics are indeed well-behaved. In particular, the scores of $\pi(\tau)$ can be, but do not have to be, learned from data~\cite{habring2026forward}. Moreover, this flexibility directly enables theoretically sound posterior
sampling, for instance in Bayesian inverse problems~\citep{habring2026forward}. In contrast, for diffusion models the conditional score is in general intractable, so that one has to resort to heuristics (we refer to~\citep{daras2024survey,gupta2024diffusion}
for surveys of such heuristics and to~\citep{zach2026statbench} for a
comprehensive practical comparison).

From a theoretical perspective, we can make another interesting observation favouring annealing in contrast to reverse diffusion sampling: It is known that, in the context of diffusion models, switching to the kinetic system does not yield any performance gains~\cite{chen2022sampling}. This is due to the fact that the kinetic forward-diffusion leads to a time-dependent score also in the velocity component. Contrary, in our approach, the target $\pi$ is varied with respect to $x$ but fixed with respect to $v$. We will see that this allows us to transfer the known advantages of underdamped Langevin\footnote{most notably the improved discretization accuracy} also to the annealed setting leading to significantly improved iteration complexity.

\paragraph{Contributions.}
In the convex setting we show convergence of \eqref{eq:an_under_sde} in the \emph{backward} Kullback-Leibler divergence\footnote{\emph{Backward} refers to the target distribution being the first argument of the $\KL$, which is in fact the more desirable variant as it penalizes modes that are present in the target but are not covered by the samples.}
\begin{equation}
    \KL(\pi_{\tau(t)}|\law (X_t)).
\end{equation}
We establish this result by introducing a new proof technique that makes tools from the
overdamped setting, previously inapplicable due to the degenerate noise,
available for the underdamped dynamics. 
Our results improve the iteration complexity to reach $\epsilon$-accuracy in total variation from $\Oc(\epsilon^{-6})$ to $\Oc(\epsilon^{-2})$~\citep{guo2025provable,cordero2025non}. Specifically, denoting the action of a curve of probability measures\footnote{the action can be interpreted as the cost of transporting \(\pi(1)\) to \(\pi(0)\) along the prescribed path}
\begin{equation}
    \Ac(\pi) \coloneqq \int_0^1|\pi'(\tau)|^2\dd \tau
\end{equation}
with $|\pi'(\tau)|$ the metric derivative of $\tau\mapsto\pi(\tau)$,
we have the following two main results where some of the parameters will be introduced below.
\begin{theorem}[continuous time, informal, see~\cref{thm:continuous,thm:continuous_strongly}]
    Let $\pot_{\tau}$ be $\lip_\tau$-smooth and $m_\tau$-convex and assume $\friction_\tau\sim \sqrt{\lip_\tau}$. Denote $\inf_\tau m_\tau = m_\infty$, $\sup_\tau \lip_\tau = \lip_\infty$. If $m_\infty=0$,~\cref{eq:an_under_sde} satisfies
    \begin{equation}
        \KL(\pi|Z_T) \lesssim \frac{\sqrt{\lip_\infty}}{T}\big(\Wc_2^2(\pi_x(1),\law(X_0)) + \Ac(\pi)\big),
    \end{equation}
    and if $m_\infty>0$ and $T\gtrsim \frac{\sqrt{\lip_\infty}}{m_\infty}$
    \begin{equation}
        \KL(\pi_{\tau}|Z_T)\lesssim \frac{\lip_\infty^{3/2}}{T^3 m_\infty^2}\Ac(\pi) + m_\infty\exp{-c\frac{m_\infty}{3\sqrt{\lip_\infty}} T} \; \Wc_2^2(\pi_x(1),\law(X_0)).
    \end{equation}
\end{theorem}

\begin{theorem}[discrete time, informal, see~\cref{lemma:disc_error,coroll:disc_complexity}]
    If $m_\infty>0$ we have for the exponential Euler discretization of~\cref{eq:an_under_sde} with step-sizes $\step_k$ denoted as $(\discX_k, \discV_k)_k$ and with $T=Nh$
    \begin{equation}
        \KL(\law(X_T)|\law(\discX_N))
            \lesssim\sum_{k=0}^{N-1}\step_{k}^3 \frac{1}{\sqrt{\lip_{\tau(t_k+1)}}}\left\{\frac{\dtbound_{\tau(t_{k+1})}^2}{T^2} (1+\frac{d}{m_{\tau(t_{k+1})}}) +\frac{\lip_{\tau(t_{k+1})}^3 d}{m_{\tau(t_{k+1})}}\right\} + \frac{\epsilon_{\score}^2}{\sqrt{\lip_1}}.
    \end{equation}
    In particular, to reach $\epsilon$ accuracy in total variation distance, the discretization of~\cref{eq:an_under_sde} requires 
    \begin{equation}
        \tilde{\Oc}\bigg(\frac{\Ac(\pi)^{1/6}\dtbound_\infty d^{1/2}}{\epsilon^{4/3}m_\infty^{5/6}} 
        \vee \frac{\Ac(\pi)^{1/2}\lip_\infty^2d^{1/2}}{m_\infty^{3/2}\epsilon^{2}}
        \vee
        \frac{\dtbound_\infty\minimizerbound^{1/2}d^{1/4}}{m_\infty^{3/4}\epsilon}
        \vee
        \frac{\lip_\infty^2\minimizerbound^{3/2}}{d^{1/4}m_\infty^{5/4}\epsilon}\bigg)
    \end{equation}
    iterations.
\end{theorem}
Despite the fact that the current analysis relies on convexity, the main focus of this work is the proof technique which admits generalizations beyond the convex setting. Indeed, convexity is used in a synchronous coupling argument, and we are currently working on a generalization to the non-convex setting by means of reflection coupling. Note that the gains in complexity heavily rely on the improved discretization accuracy due to the exponential Euler scheme in comparison to the Euler-Maruyama discretization used for overdamped annealing. This property is entirely \emph{independent of convexity}. We refer to~\cref{sec:conclusion_futurework} for an elaboration on possible extensions.

Finally, it is important to note that currently due to the time-inhomogeneity the iteration complexity is worse compared to non-annealed underdamped Langevin algorithm~\citep{zhang2023improved} which achieves $\tilde{\Oc}(\epsilon^{-1})$. This is, however, also the case for annealed overdamped Langevin~\citep{cordero2025non,guo2025provable} and the complexity gap between the annealed and non-annealed algorithms is significantly smaller for the proposed method (overdamped: $\tilde{\Oc}(\epsilon^{-6})$ to $\tilde{\Oc}(\epsilon^{-1})$; underdamped: $\tilde{\Oc}(\epsilon^{-2})$ to $\tilde{\Oc}(\epsilon^{-1})$)

\paragraph{Organization of the Article}
The remaining article is organized as follows: In~\Cref{sec:related_work} we review related work. \Cref{sec:Notation_Assumption} introduces our assumptions and notation. We present our approach and results for continuous time in~\Cref{sec:cont_time_analysis} and discrete time in~\Cref{sec:discrete_time_analysis}. Finally, we conclude with a discussion of limitations and extensions in~\Cref{sec:conclusion_futurework}.
Technical details and some proofs are deferred to~\cref{sec:postponed_proofs}.

\section{Related Work}\label{sec:related_work}
\paragraph{Non-annealed underdamped Langevin} 
Due to its accelerating properties there is a vast literature on underdamped Langevin dynamics. In the machine learning community the research was largely incited by~\citep{cheng2018underdamped}. For works considering the convex setting and synchronous coupling, see~\citep{Dalalyan2018OnSF,cheng2018underdamped,leimkuhler2024contraction}. In the non-convex setting, reflection coupling with carefully designed distance functions have been fruitfully used to derive contraction results in Wasserstein distances~\citep{schuh2024globalcontractivity,eberle2019couplings,eberle2026non_rev_lift}. Furthermore, there is a vast literature on different discretizations of underdamped Langevin dynamics. Most notably, its structure is amenable to splitting methods where individual components of the drift are integrated alternatingly~\citep{leimkuhler2013rational,leimkuhler2024contraction}. A discretization scheme that has been particularly successful is the randomized midpoint scheme~\citep{altschuler2026shiftedcompositionivballistic,shen2019randomized}. Of all the cited works, the most closely related is undoubtedly~\citep{altschuler2026shiftedcompositionivballistic}. There the authors employ the shifted Girsanov argument that is also made fruitful in the present article and leads to improved scaling in the condition number $\kappa$. 

\paragraph{Time-Inhomogeneous sampling/successive approximation}
As mentioned in the introduction we distinguish approaches of traversing and approaches of trailing the target. 
Diffusion models and flows fall into the category of traversing and constitute some of the most popular contemporary sampling techniques~\citep{albergo2025stochastic,chen2022sampling,silveri2026diffusion,papamakarios2021normalizing}. However these approaches are conceptually different from the proposed annealing and we refer to the introduction for an elaboration of settings in which annealing may be preferred.

There exist several works analyzing the behaviour of overdamped annealed Langevin, introduced in~\citep{song2019generative}. In~\citep{habring2026forward,chehab2025provable_tempering} overdamped annealing is analyzed in \emph{forward} $\KL$ relying on entropy dissipation. \citep{crucinio2026properties_geom_temp} extend on \citep{chehab2025provable_tempering} and additionally considers the time derivative of the forward \gls{kl} along the Fisher-Rao and a combined tempered Wasserstein-Fisher-Rao flow, thereby bridging the gap to \textit{sequential Monte-Carlo methods} and \textit{annealed importance sampling}.
More closely to the present article, in~\citep{cordero2025non} a Girsanov argument comparing to an ideal reference process is used to show convergence in the backward $\KL$. This approach, introduced~\citep{guo2025provable} for annealing, is also made fruitful in the present article, albeit, with non-trivial modifications due hypocoercivity.
To the best of our knowledge, annealed underdamped Langevin dynamics have only been studied in an application oriented context with no provided convergence guarantees~\citep{zilberstein2023annealedunderdamped}.

\paragraph{Shifted Girsanov} The original proof strategy of considering a coupled auxiliary process and subsequent application of Girsanov's theorem to bound the \gls{kl} was originally introduced in  \citep{arnaudon2006harnackinequalityheatkernel} in the context of contractivity for diffusion semigroups on Riemannian manifolds. Extending on this idea, similar techniques have been used in several different contexts in stochastic analysis and we refer to \citep{wang2010couplingapplications} for an overview. More recently, the works of 
\citep{altschuler2025shifted_comp_1} and  \citep{altschuler2026shiftedcompositionivballistic} used a shifted Girsanov arguments to establish reverse transport inequalities for the overdamped and underdamped Langevin dynamics. Moreover in \citep{lu2026longtimereversetransportationinequalities} the authors combine a shifted Girsanov argument with reflection coupling to obtain convergence results in \gls{kl} divergence for Langevin diffusions with possibly non-log-concave drifts.

\paragraph{Putting the present article into context}
To the best of our knowledge, there is no theoretical analysis available in the literature for annealed underdamped Langevin dynamcis. In this article we fill this gap and provide such an analysis by relying on shifted Girsanov arguments carefully adapted to the time-inhomogeneous setting. Conceptually, or work is most closely related to~\citep{altschuler2026shiftedcompositionivballistic,cordero2025non}.


\section{Notation and Assumptions}\label{sec:Notation_Assumption}

\paragraph{Notation}
Throughout this article $x$ will be used for the position variable, $v$ for the velocity, and $z=(x,v)$ for the position-velocity pair. Depending on context, these might appear as a capital or lower-case letter or in other variations ($x$, $X$, $\discX$, $\bar X$ etc.).
We use Greek letters to denote probability measures with $\pi(\tau,\dd z)\propto\exp{-\pot_\tau(x)-\frac{1}{2}|v|^2}\dd z$ the (moving) target, $\mu(t,z)$ the distribution of~\cref{eq:an_under_sde}, and $\hat\mu(t,z)$ the distribution of the discretization of~\cref{eq:an_under_sde} introduced in~\cref{sec:discrete_time_analysis}. Note here the difference in that $\pi$ depends on the parameter $\tau$ (which itself will be time-dependent), whereas $\mu$, $\hat \mu$ depend on the time $t$. Moreover, $p(\tau,z)$, $q(t,z)$, and $\hat q(t,z)$ denote the densities with respect to the Lebesgue measure of $\pi(\tau,z)$, $\mu(t,z)$, and $\hat\mu(t,z)$, respectively. We will use subscripts to denote marginal distributions, e.g., $p_x(\tau,x) = \int p(\tau,x,v)\dd v$ and analogously for the measures. Lastly, for ease of notation, we will sometimes write $\mu(t)=\mu(t,\dd z)$, $q(t)=q(t,z)$, etc.
The set of probability measures on \(\R^d\) is denoted as \(\Pc(\R^d)\). We denote with \(\mu \ll \nu\) that the probability measure \(\mu\) is absolutely continuous w.r.t. \(\nu\), and with \(\frac{\dd \mu}{\dd \nu}\) the Radon-Nikodym derivative of \(\mu\) w.r.t. \(\nu\). For a probability measure \(\mu \in \P(\R^d)\) and function \(f: \R^d \rightarrow \R^d\), \(\norm{f}_{L^2(\mu)} = (\int \norm{f}^2 \dd \mu)^\frac12\) and the \(n\)-th moment is \(\E_\mu[|\cdot|^n]\) for \(n > 0\). We define the the Wasserstein-2 distance between \(\mu, \nu \in \Pc(\R^d)\) as \(W_2(\mu, \nu) = \inf_{\gamma \in \Pi(\mu, \nu)} (\int |x - y|^2 \gamma(\dd x, \dd y))^{\frac12}\) where \(\Pi(\mu, \nu)\) is the set of probability measures on \(\R^d \times \R^d\) with marginals \((\mu)\) and \(\nu\), the total-variation distance as \(\|\mu - \nu\|_{\TV} = \sup_{A \subset \R^d}|\mu(A) - \nu(A)|\) and the Kullback--Leibler divergence as \(\KL(\mu \mid \nu)  = \int \log \frac{\dd \mu}{\dd \nu} \dd \mu\) for \(\mu \ll \nu\) and $\infty$ otherwise.
For $\Wc_2$, $\|\cdot\|_{\TV}$, and $\KL$ we will often write, \eg, $\Wc_2(X,Y)$ to denote $\Wc(\law(X),\law(Y))$ for two random variables $X,Y$.
The symbol $\tilde{\Oc}$ is used to denote complexity results omitting logarithmic factors. 

We make the following assumptions on the target, the annealing schedule and the friction parameter.

\begin{assumption}[Potential]\label{ass}\
    The potential family \(\pot_\tau \in \Cc^2(\R^d)\) satisfies
    \begin{enumerate}
    \item There exist \(0 \leq m_\tau \leq \lip_\tau\) such that $m_\tau I\preceq \nabla^2 \pot_\tau(x) \preceq \lip_\tau I$ for all \(x \in \R^d\) and \(\tau \in [0,1]\). Furthermore $\lip_\infty\coloneqq\sup_{\tau}\lip_\tau<\infty$ and $m_\infty \coloneqq \inf_\tau m_\tau\geq 0$.\label{ass:pot1} 
    We assume without loss of generality that $\lip_\tau$ and $m_\tau$ are decreasing, respectively increasing in $\tau$.\footnote{This can always be achieved by replacing $\lip_\tau$ by $\sup_{\tau'\geq \tau}\lip_{\tau'}$ and similarly for $m_\tau$.}
    \item There exists an increasing function $\tau\mapsto\dtbound_\tau>0$ such that $|\partial_\tau \nabla\pot_\tau(x)|\leq \dtbound_\tau (1+|x|)$. We denote $\beta_\infty = \sup_\tau \beta_\tau$.\label{ass:pot2}
    \item Let $x_\tau^*\in \argmin_{x}\pot_\tau(x)$. There exists $\minimizerbound>0$ such that $|\partial_\tau x^*_\tau| \leq \minimizerbound$.\label{ass:pot3}
    \end{enumerate}
\end{assumption}

\begin{assumption}[Annealing schedule]\label{ass:schedule}
The annealing schedule $\tau:[0,T]\to [0,1]$ is of the form $\tau(t) = \chi(\frac{t}{T})$ such that the following hold: 
\begin{enumerate}
    \item $\chi: [0,1]\rightarrow [0,1]$ is strictly decreasing with $\chi(0)=1$, $\chi(1)=0$\label{ass:schedule1}
    \item There exists $\chibound>0$ such that $\sup_{t\in[0,1]}\frac{|\dot \chi(t)|}{(1-t)^2}\leq \chibound$.\label{ass:schedule2}
\end{enumerate}
\end{assumption}
\begin{assumption}[Friction sequence]\label{ass:friction}
    There exists $1= \tau_0>\tau_1>\dots >\tau_n=0$ such that $\friction_{\tau} = \friction_k$ for $\tau\in (\tau_k,\tau_{k+1}]$ and there exists $\dtfric>0$ such that $\sup_k\frac{|\log \friction_{k+1} - \log\friction_k|}{\tau_{k+1}-\tau_k}\leq \dtfric$. We denote $\friction_\infty = \sup_k \friction_k$.
\end{assumption}
\begin{remark}[High friction regime]
    Note that we will later fix the choice $\friction_k = 2\sqrt{\lip_{\tau_k}}$. In this case we have for the variance preserving diffusion model where $\lip_\tau = \frac{\lip_0}{1+\tau(\lip_0-1)}$ (\cf\,\cref{lemma:tau_dep_hessian_bounds})
    \begin{equation}
        \partial_\tau \log\sqrt{\lip_\tau} = \frac{\tau (\lip_0 - 1)}{1+ \tau (\lip_0 - 1)}.
    \end{equation}
    and by monotonicity of this derivative $\nu_k \lesssim \lip_0$.
\end{remark}
\begin{example}[Variance preserving convolution]\label{example:conv}
    Let $\pot_0$ be $m_0$-strongly convex and $\lip_0$-smooth and assume $\|\nabla^3\log p(0,x)\|_{op}\leq M$ for some $M>0$ and all $x$. Then all assumptions are satisfied for the variance preserving convolutional path defined for $\kappa>0$ and $\tau\in [0,1-\kappa]$
    \begin{equation}
        \pi_\tau \coloneqq \law(\sqrt{1-\tau}X + \sqrt{\tau}Z), \quad X \perp Z,\; X\sim \pi_0,\; Z\sim \Nc(0,\id)
    \end{equation}
    with the schedule defined by $\chi(t) = \big(\frac{1+\cos(\pi t)}{2}\big)^2$ or simply $\chi(t) = (1-t)^3$. See~\cref{sec:convolutional} for a proof.
\end{example}
\begin{remark}
    The offset $\kappa>0$ in \cref{example:conv} is necessary since the
    variance-preserving path on the finite horizon $[0,1]$ degenerates as
    $\tau\to 1$. The resulting minor inconsistency with the convention
    $\tau(0)=1$ can be removed by a simple rescaling of time.
\end{remark}

\section{Continuous Time Analysis.}\label{sec:cont_time_analysis}
In this section, we analyze \cref{eq:an_under_sde} in continuous time. The main idea to proof convergence is to compare the annealed Langevin dynamics to an ideal reference process following exactly $\pi(\tau)$.

\subsection{Warmup: The overdamped annealed Langevin}
To motivate the proposed approach we begin by briefly recalling the proof strategy for the overdamped case employed in \citep{guo2025provable,cordero2025non}. Consider the overdamped annealed Langevin diffusion
\begin{equation}\label{eq:overdamped_annealing}
    \dd X_t = -\nabla \pot_{\tau(t)}(X_t)\dd t + \sqrt{2}\dd W_t.\tag{ANOLD}
\end{equation}
By \Cref{thm:optimal_velocity} there exists a velocity field \(\hat b: [0,T] \times \R^d \rightarrow \R^{2d}\) such that the curve of probability measures \(t \mapsto \pi(\tau(t))\) on \(\R^{2d}\) satisfies the \textit{continuity equation} \(\partial_t \pi = \nabla \cdot (-\hat b\pi)\) or equivalently
\begin{equation}\label{eq:target_process_OLD}
        \dd \hat X_t = \hat{b}(t, \hat X_t)\dd t, \quad \hat{X_0} \sim \pi_X(1)
\end{equation}
satisfies $\law(\hat X_t) = \pi(\tau(t))$ for all $t$ and \(\hat b\) may be chosen such that \(|\hat b|_{L^2(\pi_x(\tau(t)))} = |(\pi_x \circ \tau)'|\) for a.e. \(t\). 
Noting that $\partial_t \pi = \nabla\cdot((-\hat b - \nabla\log \pi)\pi) + \Delta \pi$ 
where we added and subtracted $\Delta\pi$ and used that $\hat X_t \sim \pi(\tau(t))$, we find the following \gls{sde}, equivalent in distribution to~\cref{eq:target_process_OLD}
\begin{equation}\label{eq:target_process_OLD2}
    \dd \hat X_t = (\hat b(t,\hat X_t) + \nabla \log \pi(\tau(t),X_t))\dd t + \sqrt{2}\dd W_t.
\end{equation}
Comparing to~\cref{eq:overdamped_annealing}, applying the data processing inequality \citep[Theorem 1.5.6]{Chewi26Book} and Girsanov's theorem in the form of~\cref{coroll:girs_KL} yields
\begin{equation}
    \begin{aligned}
        \KL(\hat{X}_t|X_t)
        \leq\E\left [\frac{1}{4}\int _0^s|\hat b_x(s,\hat X_s)|^2|\dd s\right ]
        =\frac{1}{4} \int _0^s|(\pi\circ\tau)'(t)|^2\dd s
        =&\frac{1}{4} \int _0^s|\dot\tau|^2|(\pi'(\tau(t))|^2\dd s\\
        \leq&\frac{\|\dot\tau\|_\infty}{4} \Ac(\pi)
    \end{aligned}
\end{equation}
which tends to zero as $\|\dot\tau\|_\infty\to 0$. 

We will now show how to adapt this argument to the underdamped setting, which turns out to be non-trivial, due to the degenerate diffusion, which prohibits a direct application of Girsanov's theorem.

\subsection{Underdamped annelaed Langevin}
\subsubsection{Setup and auxiliary process}
In lifted space we work with the path $\pi(\tau)$ satisfying $ \pi(\tau) = \pi_x(\tau,\dd x)\otimes\pi_v(\tau, \dd v)$ with $\frac{\pi_x(\tau)}{\dd x}\propto \exp{-\pot_\tau(x)}$ and $\pi_v(\tau)$ a standard Gaussian for all $\tau$. In particular, the $x$ and $v$ components are independent and we have $\nabla_x\log \pi(\tau,(x,v)) = -\nabla \pot_\tau(x)$ and $\nabla_v \log \pi(\tau,(x,v)) = v$.
As shown in \cref{lemma:optimal_velocity_independent}, in this case, the optimal velocity field $\hat b$ factorizes as well, that is, $\hat b = (\hat b_x, \hat b_v)$ with $\hat b_x, \hat b_v:[0,T]\times \R^d\rightarrow \R^d$ the velocity fields for $\pi_x, \pi_v$, respectively. Moreover, since $\pi_v$ is constant, $\hat b_v \equiv 0$. We will thus drop the index and denote simply $\hat b = \hat b_x$. We obtain for the underdamped reference process
\begin{equation}
    \begin{cases}
        \dd \hat X_t =& \hat b(t,\hat X_t) \dd t \\
        \dd \hat V_t =& 0\dd t,
    \end{cases}
\end{equation}
with Fokker-Planck equation $\partial_t \pi = \nabla_x \cdot(-\hat b \pi)$. This Fokker-Planck equation is equivalent to
\begin{equation}
    \partial_t \pi = \nabla\cdot\begin{pmatrix}
        \begin{bmatrix}
            -\hat b - v\\
            \friction_{\tau(t)} v - \nabla_x \log \pi
        \end{bmatrix}\pi
    \end{pmatrix}
    + \friction_{\tau(t)} \Delta_v\pi 
\end{equation}
which is obtained by adding and subtracting $\friction_{\tau(t)}\Delta \pi$ and $\nabla_x \nabla_v \pi$ and noting that $v = -\nabla_v \log\pi$.
Thus, we may equivalently consider the reference \gls{sde}
\begin{equation}\label{eq:reference_underdamped}
    \begin{cases}
        \dd \hat X_t &= \{\hat b(t, \hat X_t) + \hat V_t\} \dd t\\
        \dd \hat V_t &= \{-\friction_{\tau(t)} \hat V_t - \nabla \pot_{\tau(t)}(\hat X_t) \}\,\dd t + \sqrt{2 \friction_{\tau(t)}}\dd W_t.\\
    \end{cases}\tag{REF}
\end{equation}
Note that Girsanov's theorem is not applicable due to the discrepancy in the $x$-component of the drift between~\eqref{eq:reference_underdamped} and~\eqref{eq:an_under_sde}. The main idea to circumvent this issue is the following: For any process $(Z^\aux{}_t)_t$ which satisfies $Z^\aux{}_T \overset{d}{=} \hat Z_T$ we have $\KL(\hat Z_T|Z_T) = \KL(Z^\aux{}_T|Z_T)$. In order to make sure Girsanov is applicable to compute $\KL(Z^\aux{}_T|Z_T)$, we only have to set the drift in the $x$-component of $Z^\aux{}_T$ to be identical to that of~\eqref{eq:an_under_sde}.
We, therefore, define for parameters $\eta^x_t, \eta^v_t$ satisfying $\eta^x_t,\eta^v_t\rightarrow \infty$ as $t\to T$
\begin{equation}\label{eq:defin_aux_ref}
    \begin{aligned}
    &\begin{cases}
        \dd X_t^{\aux{}} &= V_t^{\aux{}}\dd t \\
        \dd V_t^{\aux{}} &=
        \begin{aligned}[t]
             &\{-\friction_{\tau(t)} V_t^{\aux{}}
            - \nabla_x \pot_\tau(X_t^{\aux{}}) \\
            & + \eta_t^x(\hat X_t - X_t^{\aux{}})
            + \eta_t^v(\hat V_t - V_t^{\aux{}})\}\dd t
            + \sqrt{2 \friction_{\tau(t)}}\,\dd W_t.
        \end{aligned}\\
    \end{cases}
    \end{aligned}\tag{AUX}
\end{equation}
The precise definition of $\eta^x_t,\eta^v_t$ will be given below.
This process constitutes a \textit{bridge} between its initial distribution and $(\hat X_T, \hat V_T)$. Crucially, we define the auxiliary process to be synchronously coupled with the reference, that is, the Brownian motion in~\eqref{eq:defin_aux_ref} is identical to that in~\eqref{eq:reference_underdamped}.

It is now to be expected that we may apply Girsanov's theorem to bound the difference between~\eqref{eq:an_under_sde} and~\eqref{eq:defin_aux_ref} as both processes admit the exact same drift in the $x$-component. Indeed, this is the case, as elaborated in~\cref{sec:girsanov_rigorous} and we obtain
\begin{equation}\label{eq:KL-aux-ud}
    \begin{aligned}
        \KL(\hat Z_T|Z_T) = \KL(Z^\aux{}_T|Z_T)\leq
         \E_{\P}\left [ \int_0^T \frac{1}{4\friction_{\tau(s)}}|\eta_s^x(\hat X_s - X_s^{\aux{}})
            + \eta_s^v(\hat V_s - V_s^{\aux{}})|^2\dd s\right].
    \end{aligned}
\end{equation}

\begin{remark}
    One may also consider an auxiliary process constructed such that $Z^\aux{}_T \overset{d}{=} Z_T$ and subsequently bound $\KL(\hat Z_T|Z_T) = \KL(\hat Z_T|Z^\aux{}_T)$. This, in turn, requires a bound on \(\E[|b_t|^2]\), which can be obtained by assuming the vector field is \(L_\tau\)-Lipschitz continuous.
\end{remark}

\subsubsection{Convergence analysis}
Our goal is to bound the right-hand side in~\eqref{eq:KL-aux-ud}.
We proceed by introducing some additional notation, customary when working with underdamped systems. First, collecting the velocity terms in~\eqref{eq:defin_aux_ref} motivates the following definitions
\begin{equation}
    \hat{\friction}_t = \friction_{\tau(t)} + \eta^v_t.
\end{equation}
Moreover, we use the following coordinate transformation
\begin{equation}
    A_t \coloneqq\begin{bmatrix}
        1&0\\
        1&\frac{2}{\hat{\friction}_t}
    \end{bmatrix},\quad 
    \begin{bmatrix}
        x\\
        p
    \end{bmatrix}
    \coloneqq A_t z
    = 
    \begin{bmatrix}
        x\\
        x+\frac{2}{\hat{\friction}_t}v
    \end{bmatrix}.
\end{equation}

Inspired by~\citep{altschuler2026shiftedcompositionivballistic} we defined the shift parameters $\eta^x_t, \eta^v_t$ as follows:
\begin{assumption}[Shift Parameters]\label{ass:shifts}
The shift parameters are defined according to one of the following to alternatives where the parameter $c_0>0$ will be chosen later.
\begin{enumerate}
        \item $\eta_t^v = \frac{c_0\omega}{\exp{\omega(T-t)} - 1}$ and $\eta_t^x = \frac{\hat{\friction}_t\eta_t^v}{2}$ where $\omega=\frac{m_{\infty}}{3\friction_{\infty}}$ in the case that $m_\infty>0$.\label{item_shifts_strongly}
        \item $\eta_t^v = \frac{c_0}{T-t}$ and $\eta_t^x = \frac{\hat{\friction}_t\eta_t^v}{2}$ in the case that $m_\infty=0$.\label{item_shifts_non-strongly}
    \end{enumerate}
\end{assumption}
Note that the second alternative is exactly the limit of the first as $\omega\to 0$. We obtain the following contraction result, whose proof is deferred to~\cref{sec:proof_contraction}.

\begin{lemma}[Contraction between auxiliary and reference]\label{lemma:contraction}
    Denote the errors between auxiliary and reference process in transformed coordinates as
    \begin{equation}\label{eq:shifted coordinates}
        T_t \coloneqq \begin{bmatrix}
                \hat X_t - X_t^\aux{}\\
                \hat P_t - P_t^\aux{}
    \end{bmatrix}.
    \end{equation}
    Then it holds true that 
    \begin{equation}\label{eq:shift_error_differential}
        \begin{aligned}
            \dd \norm{T_t}^2 
            = -2 \inner{T_t}{Q_tT_t} \dd t
            + 2\inner{T_t}{\begin{bmatrix} b_t\\b_t
    \end{bmatrix}} \dd t
        \end{aligned}
    \end{equation}
    where 
    \begin{equation*}
    Q_t = \begin{bmatrix}
        \frac{\hat{\friction}_t}{2}I & (\frac{\eta_t^x}{\hat{\friction}_t} - \frac{\hat{\friction}_t}{2} - \frac{\dot \hat{\friction}_t}{2\hat{\friction}_t})I + \frac{1}{\hat{\friction}_t}H_t \\
        (\frac{\eta_t^x}{\hat{\friction}_t} - \frac{\hat{\friction}_t}{2} - \frac{\dot \hat{\friction}_t}{2\hat{\friction}_t})I + \frac{1}{\hat{\friction}_t}H_t & (\frac{\hat{\friction}_t}{2} + \frac{\dot \hat{\friction}_t}{\hat{\friction}_t})I
    \end{bmatrix}.
\end{equation*}
and a random process $H_t\in \R^{d\times d}\succeq m$ given in the proof.
\end{lemma}
This contraction result together with~\cref{eq:KL-aux-ud} allows us to prove the following two convergence results for the continuous-time annealed underdamped dynamics. The proofs can be found in~\cref{sec:proof_cont_convex,sec:proof_cont_strongly}
\begin{lemma}[Convergence in Backward KL in the convex case]\label{thm:continuous}
    Assume $m_\infty\geq 0$ and the shift parameters are chosen according to~\cref{ass:shifts}, \cref{item_shifts_non-strongly}. Recall that $\hat Z_t\sim\pi(\tau(t))$ denotes the reference process~\eqref{eq:reference_underdamped} and $Z_t$ the annealed underdamped Langevin process~\eqref{eq:an_under_sde}. Then it holds true that
    \begin{equation}
        \KL(\hat Z_T|Z_T) \lesssim \bigg(\frac{\friction_\infty}{T} + \frac{1}{ T^3}\bigg)\big(\Ac(\pi) + \Wc^2_2(\hat Z_0, Z_0)\big).
    \end{equation}
\end{lemma}

\begin{lemma}[Convergence in Backward KL in the strongly convex case]\label{thm:continuous_strongly}
    Let $m_\infty>0$, that is, $\pot_\tau$ is uniformly strongly-convex and let $\eta^x_t, \eta^v_t$ be chosen according to~\cref{ass:shifts}, \cref{item_shifts_strongly}. Then,
    \begin{equation}
        \KL(\hat Z_T|Z_T)\lesssim \frac{(\friction_\infty+\omega^2)}{\omega^2 T^3}\Ac(\pi) + \frac{\friction_\infty\omega + \omega^3}{(1-\exp{-\omega T})^{4}}\exp{-c'\omega T} \; \Wc^2_2(\hat Z_0, Z_0).
    \end{equation}
\end{lemma}
\begin{remark}
    The gains in the strongly convex case are two-fold: (i) The initialization error caused by $T_0$ decays now exponentially in $T$ instead of $T^{-1}$. (ii) The annealing error decays as $T^{-3}$ instead of $T^{-1}$.
\end{remark}
\begin{remark}\label{rem:finite_action}
    The results in \Cref{thm:continuous,thm:continuous_strongly} hold under a finite action assumption \(\Ac(\pi) < \infty\). 
    Lemma 3 in~\citep{guo2025provable} provides expressions for upper bounds of the action under a \gls{lsi} or \gls{pi} along the curve \(t \mapsto \pi(\tau(t))\).
    In~\citep{cordero2025non} the authors show that for the variance preserving convolution path it holds $\Ac(\pi)\lesssim d\vee M_2$ with $M_2 = \E_{\pi(0)}[|\cdot|^2]$ the second moment of the target.
\end{remark}

\section{Discrete Time Analysis}\label{sec:discrete_time_analysis}
We now turn to analyzing a discretized version of the proposed algorithm. To account for potentially learned scores we assume access to a score function
\begin{equation}
    s_\theta(\tau,x)\approx -\nabla \pot_{\tau}(x).
\end{equation}
We impose the following assumption on the score accuracy which is identical to that in~\citep{cordero2025non}.
\begin{assumption}[Score approximation]\label{ass:score}
    The approximate score $s_\theta(\tau,x)$ satisfies
    \begin{equation}
        \sum_{k=0}^{N-1}\step_k\E[|s_{\theta}(\tau(t_k),X_{t_k}) + \nabla\pot_{\tau(t_k)}(X_{t_k})|^2]<\epsilon_\score^2.
    \end{equation}
\end{assumption}
\begin{remark}
    Note that contrary to the literature on diffusion models~\citep{silveri2026diffusion,chen2022sampling} in annealed Langevin the score accuracy is required integrated over time~\citep{cordero2025non}. 
    However, we can make the following observation. Let us define $\pot^\theta_\tau$ via $\nabla \pot^\theta_\tau(x) = -s_\theta(\tau,x)$.\footnote{assuming integrability of the learned score for simplicity} Then,~\eqref{eq:discretization} is an exact-score discretization of~\eqref{eq:an_under_sde} for $\pot^\theta$ and it would be sufficient to estimate the bias between $\exp{-\pot^\theta_0}$ and $\exp{-\pot_0}$ without considering errors accumulated for $\tau>0$. This could be done using, \eg, techniques from~\citep{renaud2025sampling}.
\end{remark}

As a discrete algorithm we consider the stochastic exponential-Euler discretization \citep{cheng2018underdamped,zhang2023improved} which is amenable to Girsanov.
More precisely, let $(\step_k)_k$, $\step_k>0$ be a sequence of step sizes and $t_k = \sum_{\ell<k}\step_k$. We define the stochastic process $(\discX_t,\discV_t)_t$ piecewise for $t\in(t_k,t_{k+1}]$ as
\begin{equation}\label{eq:discretization}
    \begin{cases}
        \dd \discX_t &= \discV_t \dd t\\
        \dd \discV_t &= -\friction_k \discV_t \dd t + s_\theta(\tau(t_k),\discX_{t_k})\dd t + \sqrt{2\friction }\dd W_t^v.
    \end{cases}
    \tag{EE}
\end{equation}

More specifically, while discretizing the force $\nabla\pot_\tau$, we integrate $\discV_t$ exactly over the intervals $(t_k,t_{k+1}]$. Note that one step of this discretization is an Ornstein-Uhlenbeck process which can be solved exactly, \cf~\cref{sec:ee_discretization_explicit}. The proof of the following discretization error result can be found in~\cref{sec:proof_disc}.

\begin{lemma}[Discretization Error]\label{lemma:disc_error}
    Assume that $\friction_k = 2\sqrt{\lip_{\tau(t_k+1)}}$ and $T>\frac{\minimizerbound\sqrt{\lip_\infty}}{\sqrt{d m_\infty}}$. Then it holds true that 
    \begin{equation}
        \KL(Z_t|\discZ_T)
            \lesssim\sum_{k=0}^{N-1}\step_k^3 \frac{1}{\friction_{k}}\left\{\frac{\dtbound_{\tau(t_{k+1})}^2}{T^2} (1+\frac{d}{m_{\tau(t_{k+1})}}) +\frac{\lip_{\tau(t_{k+1})}^3 d}{m_{\tau(t_{k+1})}}\right\} + \frac{\epsilon_{\score}^2}{\friction_0}
    \end{equation}
\end{lemma}

\begin{coroll}[Complexity]\label{coroll:disc_complexity}
    In the setting of~\cref{lemma:disc_error} with constant step size $\step_k = \step$ for all $k$. If $\epsilon_\score\lesssim \sqrt{\friction_0}\epsilon$, to reach $\epsilon$ accuracy in total variation distance requires 
    \begin{equation}
        \tilde{\Oc}\bigg(\frac{\Ac(\pi)^{1/6}\dtbound_\infty d^{1/2}}{\epsilon^{4/3}m_\infty^{5/6}} 
        \vee \frac{\Ac(\pi)^{1/2}\lip_\infty^2d^{1/2}}{m_\infty^{3/2}\epsilon^{2}}
        \vee
        \frac{\dtbound_\infty\minimizerbound^{1/2}d^{1/4}}{m_\infty^{3/4}\epsilon}
        \vee
        \frac{\lip_\infty^2\minimizerbound^{3/2}}{d^{1/4}m_\infty^{5/4}\epsilon}\bigg)
    \end{equation}
    iterations.
\end{coroll}

Let us comment on the derived complexity. We achieve $\tilde{\Oc}(\epsilon^{-2})$ in comparison to $\tilde{\Oc}(\epsilon^{-6})$ in the overdamped case~\citep{guo2025provable,cordero2025non}. While the current analysis relies on strong convexity (we will comment on this in more detail in a moment), the improvement by four orders of magnitude is nonetheless substantial. We highlight again, that such improvements are not achievable for kinetic diffusion models~\citep{chen2022sampling} since in this case the score is also time-dependent in the $v$-component. Crucially, because the annealing acts only in the $x$-component we are able to show in our setting, that the increased discretization accuracy carries over as in the setting without annealing. 
The iteration complexity of \cref{eq:an_under_sde} is investigated in \cref{fig:complexity}.

\begin{figure}
    \centering
    \includegraphics[width=0.8\linewidth]{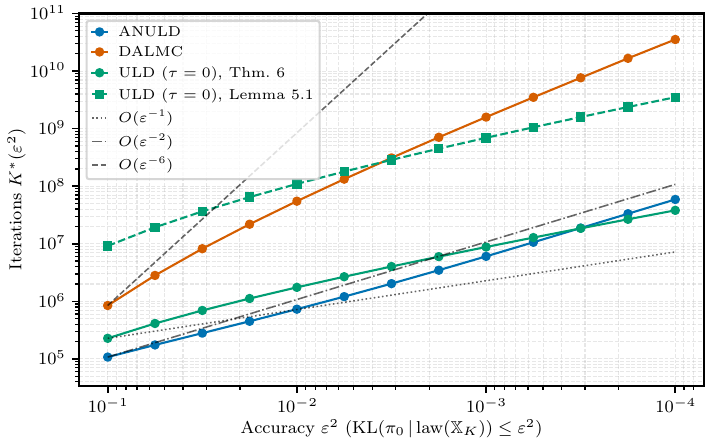}
    \caption{Iteration complexity observed when annealing an ill-conditioned Gaussian target.
    Every method uses the step size prescribed by its own theory. ULD (no annealing) is shown with the step size of the best known bound in the strongly-convex case~\cite[Theorem 6]{zhang2023improved} (solid) and with the step size of our own analysis (dashed), which is also used by \cref{eq:an_under_sde}. Black lines indicate the theoretical rates. ANULD requires fewer iterations than DALMC~\cite{cordero2025non} and ULD, when using the parameters from our analysis. Only at high accuracies does ULD with the sharper bound overtake ANULD. Details of the experimental setup and the step size rules are given in \cref{sec:complexity}.}
    \label{fig:complexity}
\end{figure}


\section{Conclusion, limitations and future work}\label{sec:conclusion_futurework}
In this article, we presented the first quantitative convergence analysis of
annealed underdamped Langevin sampling. Since the degenerate Brownian motion
renders the Girsanov-based tools from the overdamped setting inapplicable, we
developed a novel shifted Girsanov argument based on a carefully designed auxiliary
process. Combined with the higher accuracy of the exponential Euler
discretization, this yields an iteration complexity of
$\tilde{\Oc}(\epsilon^{-2})$ to reach $\epsilon$-accuracy in total variation,
compared to $\tilde{\Oc}(\epsilon^{-6})$ for annealed overdamped Langevin.

\paragraph{Limitations}
We identify two main limitations in the current analysis. The first one is that our results are only in total variation, not directly in $\KL$. This is due to relying on the triangle inequality for the total variation norm. Secondly, our results currently are restricted to the strongly convex setting. It is worth noting, however, that the analysis can immediately be extended to the merely convex case $m_\infty=0$. Indeed, the continuous-time convergence is already proven in this case. The only gap to fill is the derivation of moment bounds which are easy to obtain also under weaker conditions such as dissipativity or even weaker growth conditions on $\nabla \pot$, \cf~\citep{durmus2021uniform}.

\paragraph{Future work}
In future work we plan on effectively removing all the above mentioned limitations. We already have a concrete plan how to achieve this. 
(i) Convexity: The requirement of convexity emerges from the used synchronous coupling to obtain contraction of the underdamped Langevin dynamcis in~\cref{lemma:contraction}. There is, however, vast literature on contraction of underdamped Langevin also in non-convex settings by relying instead on reflection couplings~\citep{eberle2017coupling_kinetic,eberle2019couplings,schuh2024globalcontractivity}. This has already been made fruitful in~\citep{lu2025long}.
(ii) $\KL$-bound in discrete time: To obtain directly results in $\KL$ the strategy will be apply the shifted Girsanov argument also in the discretized setting instead of relying on the triangle inequality bound as is done in the current article. Finally, we plan to conduct numerical experiments for Bayesian imaging inverse problems.

\paragraph{AI use}
The conception and ideas for the paper came solely from the authors. During the writing the authors faced some technical issues for which Claude provided a solution. However, this solution was then entirely discarded in favour of a more elegant one found by the authors. Other than that Claude was used to find related literature, to conduct a final review of the paper and to improve sentence structure and formulations in the written text.

\bibliographystyle{plainnat}
\bibliography{references}

\appendix

\section{Proofs of the main results}\label{sec:postponed_proofs}
\subsection{Continuous Time}
\subsubsection{Rigorous arguments for the application of Girsanov's theorem}\label{sec:girsanov_rigorous}
Denote the probability space in which ~\eqref{eq:an_under_sde} and \eqref{eq:defin_aux_ref} are defined as $(\Omega,\Fc,\P)$ and recall that we have
\begin{equation}\label{eq:defin_aux_ref}
    \begin{aligned}
    &\begin{cases}
        \dd \hat X_t &= \{\hat b(t, \hat X_t) + \hat V_t\} \dd t\\
        \dd \hat V_t &= -\friction_{\tau(t)} \hat V_t \dd t- \nabla \pot_{\tau(t)}(\hat X_t) \,\dd t + \sqrt{2 \friction_{\tau(t)}}\dd W_t\\
        \dd X_t^{\aux{}} &= V_t^{\aux{}}\dd t \\
        \dd V_t^{\aux{}} &=
        \begin{aligned}[t]
             &\{-\friction_{\tau(t)} V_t^{\aux{}}
            - \nabla_x \pot_\tau(X_t^{\aux{}}) \\
            & + \eta_t^x(\hat X_t - X_t^{\aux{}})
            + \eta_t^v(\hat V_t - V_t^{\aux{}})\}\dd t
            + \sqrt{2 \friction_{\tau(t)}}\,\dd W_t.
        \end{aligned}
    \end{cases}
    \end{aligned}\tag{AUX}
\end{equation}
We assume that this \gls{sde} admits a unique strong solution.
Denote $U_t = -\eta_t^x(\hat X_t - X_t^{\aux{}}) - \eta_t^v(\hat V_t - V_t^{\aux{}})$, assuming Girsanov's theorem to be applicable~\citep[Section 3.5]{shreve1991brownian} the process
\begin{equation}
    \tilde W_t \coloneqq W_t - \int_0^s \frac{1}{\sqrt{2\friction_{\tau(t)}}}U_s\dd s
\end{equation}
is a Brownian motion in the probability space $(\Omega,\Fc,\Q)$ where $\Q$ is defined via
\begin{equation}
    \frac{\dd \Q}{\dd \P} = M_t((\hat Z_s, Z^\aux{}_s)_{s\leq t}) \coloneqq\exp{\int_0^t \frac{U_s}{\sqrt{2\friction_s}} \cdot\dd W_s - \int_0^t\frac{|U_s|^2}{4\friction_s}\dd s}
\end{equation}
with $M_t: C([0,t],\R^{4d}|)\rightarrow [0,\infty)$ a deterministic function.\footnote{Note that $(W_t)_t$ can indeed be expressed in terms of $\hat Z$, $Z^\aux{}$ allowing to define $M_t$ as a function only of $\hat Z$, $Z^\aux{}$.}
Then, replacing $(W_t)_t$ in the definition of \eqref{eq:defin_aux_ref} by $(\tilde W_t)_t$ shows that $(\hat Z, Z^\aux{})$ solves $\Q$-a.e. the \gls{sde}
\begin{equation}\label{eq:an_under_sde_backward_KL}
\begin{aligned}
    &\begin{cases}
        \dd \tilde X_t &= \{\tilde b(t, \tilde X_t) + \tilde V_t\} \dd t\\
        \dd \tilde V_t &= \{-\friction_{\tau(t)} \tilde V_t - \nabla \pot_{\tau(t)}(\tilde X_t) -\eta_s^x(\tilde X_s - X_s) - \eta_s^v(\tilde V_s - V_s)\} \,\dd t + \sqrt{2 \friction_{\tau(t)}}\dd W_t\\
        \dd X_t &= V_t \, \dd t\\
        \dd V_t &= \{-\friction_{\tau(t)} V_t - \nabla_x\log \pi(\tau(t),X_t)\} \, \dd t + \sqrt{2\friction_{\tau(t)} } \, \dd W_t.
    \end{cases}
\end{aligned}
\tag{AN-KLD}
\end{equation}
It is well-known that
\begin{equation}
    \frac{\dd \law( (\hat{Z}_s,Z^\aux{}_s)_{s\leq t})}{\dd \law( (\tilde{Z}_s,Z_s)_{s\leq t})} = M_t^{-1}.
\end{equation}
Indeed, for an arbitrary measurable function $f:\R^{4d}\to \R$ it holds
\begin{equation}
    \begin{aligned}
        \int f \dd \law(\tilde Z, Z)
        = \E_{\Q}[f(\hat Z, Z^\aux{})] = \E_{\P}[f(\hat Z, Z^\aux{}) M]
        = \int f M \dd \law(\hat Z, Z^\aux{})
    \end{aligned}
\end{equation}
Then it follows for the $\KL$ divergence via applying the data processing inequality twice
\begin{equation}
    \begin{aligned}
        \KL(Z^\aux{}_t|Z_t)\leq \KL((Z^\aux{}_s)_{s\leq t}|(Z_s)_{s\leq t})
        \leq& \KL((\hat Z_s, Z^\aux{}_s)_{s\leq t}|(\tilde Z_s,Z_s)_{s\leq t})\\
        =& \E_{\P}[-\log M_t]\\
        =& \E_{\P}\left [ \int_0^t \frac{|U_s|^2}{4\friction_{\tau(s)}}\dd s\right].
    \end{aligned}
\end{equation}

\subsubsection{Proof of~\cref{lemma:contraction}}\label{sec:proof_contraction}
\begin{proof}
Denote the discrepancy in original coordinates as $D_t = \hat Z_t - Z^\aux{}_t$.
By the synchronous coupling we have
\begin{align} \label{eq:ito_sq_T_t}
    \dd \norm{T_t}^2 = 2\langle T_t,  \dd T_t \rangle
\end{align}
and 
\begin{align} \label{eq:difference_sde_transformed}
    \dd T_t = \dd (A_t D_t) 
    &= \dot A_t D_t \dd t  + A_t \dd D_t.
\end{align}
For the second term we find,
\begin{align}
    \dd D_t &= 
    \begin{bmatrix}
        0 & I \\
        -\eta_t^x I - H_t & - \hat{\friction}_t I  
    \end{bmatrix}
    D_t \dd t 
    + \begin{bmatrix}
        b(t,\hat X_t) \\
        0
    \end{bmatrix} \dd t\\
    &= M_t D_t \dd t + \begin{bmatrix}
        b(t,\hat X_t) \\
        0
    \end{bmatrix} \dd t
\end{align}
where we used that by the fundamental theorem of calculus
\begin{equation}
    \nabla_x \pot_\tau(t)(\hat X_t) - \nabla_x \pot_\tau(t)(X_t^\aux{}) = \underbrace{\int_0^1\nabla^2 \pot_{\tau(t)}(X_t^\aux{} + \lambda (\hat X_t-X^\aux{}_t) \dd \lambda}_{\eqqcolon H_t}\, (\hat X_t - X_t^\aux{}).
\end{equation}
Plugging into \cref{eq:ito_sq_T_t} yields
\begin{align*}
    \dd \norm{T_t}^2 
    =& 2 \inner{T_t}{(\dot A_t + A_t M_t)A_t^{-1} T_t} \dd t
    + 2\inner{T_t}{\begin{bmatrix} b_t\\b_t
    \end{bmatrix}} \dd t
\end{align*}
where we recognize $Q = -\mathrm{sym}((\dot A_t + A_t M_t)A_t^{-1})$ with the symmetrization $\mathrm{sym}(M) \coloneqq (M+M^\top)/2$ for any quadratic matrix $M$.
\end{proof}

\subsubsection{Proof of~\cref{thm:continuous}}\label{sec:proof_cont_convex}
We first state the following bound on the eigenvalues of $Q_t$
\begin{lemma}[{\citealt[Lemma 3.14]{altschuler2026shiftedcompositionivballistic}}]\label{lemma_eigenvalue_bound}
    If $c_0\geq 24$, it holds
    \begin{equation}
        \lambda_{\min}(Q_t)\geq \frac{\omega}{2} + \frac{\eta^v_t}{48}.
    \end{equation}
\end{lemma}

\begin{lemma}\label{lemma:T_bound_varying_fric}
    Assume $T\gtrsim \dtfric$. It holds for some absolute constant $c'>0$
    \begin{equation}
        \begin{aligned}
            \E[|T_t|^2]
            \lesssim& \bigg(\frac{T-t}{T}\bigg)^{c_0c'}
            \E[|T_0|^2] + \int_0^t \frac{(T-t)^{c'c_0}}{(T-s)^{c_0c'-1}}\E[|b_s|^2]\dd s\\
        \end{aligned}
    \end{equation}
\end{lemma}
\begin{proof}
    By~\cref{lemma_eigenvalue_bound} we have $\lambda_{\min}(Q_t)\geq \frac{\eta^v_t}{48}$. Applying Young's inequality to~\eqref{eq:shift_error_differential} then yields
    \begin{equation}
        \begin{aligned}
            \dd |T_t|^2 
            =& - \lambda_{\min}(Q_t)|T_t|^2 + 2\frac{|b_t|^2}{\lambda_{\min}(Q_t)} \dd t\\
            \leq& - \frac{\eta^v_t}{48}|T_t|^2 + 92\frac{|b_t|^2}{\eta^v_t} \dd t.
        \end{aligned}
    \end{equation}
    Taking expectations and using Grönwall's inequality and the definition of $\eta^v_t$ leads to
    \begin{equation}\label{eq:KL_boundproof3}
        \begin{aligned}
            \E[|T_{t_{k+1}}|^2]
            \lesssim& \exp{-c'\int_{t_k}^t\lambda_{\min}(Q_s)\dd s}
            \E[|T_{t_k}^+|^2] + \int_{t_k}^{t_{k+1}} \exp{-c'\int_{s}^t\lambda_{\min}(Q_r)\dd r}\E[|b_s|^2]\dd s\\
        \end{aligned}
    \end{equation}
    for some $c'>0$. Note that in the grid edges $(t_k)_k$, the transformation matrix $A_t$ has a discontinuity and, thus, we define $A_{t_k}^+ = \lim_{t\downarrow t_k}A_t$, $z^+_{t_k} = A_{t_k}^+ z_{t_k}$ and analogously $T_{t_k}^+$. Then one can check that 
    \begin{equation}
        A_{t_k}^+ A_{t_k}^{-1} = \id + (1-r_k)\begin{bmatrix}
            0&0\\
            1&-1
        \end{bmatrix}
    \end{equation}
    where $r_k = \frac{\friction_{k} + \eta^v_{t_k}}{\friction_{k+1} + \eta^v_{t_{k}}}$. Then we have for the spectral norm 
    \begin{equation}
        \|A_{t_k}^+ A_{t_k}^{-1}\|
        \leq 1+\sqrt{2}(1-r_k)
        \leq \exp{\log(1+\sqrt{2}(1-r_k))}
        \leq \exp{\sqrt{2}(1-r_k)}.
    \end{equation}
    Moreover, since for $0<a<b$, $x>0$ we have $\frac{a+x}{b+x}\geq \frac{a}{b}$ it holds that $r_k\geq \frac{\friction_k}{\friction_{k+1}}$ and we have
    \begin{equation}
        1-r_k
        = 1-\exp{\log \frac{\friction_k}{\friction_{k+1}}}
        \leq 1-\exp{-\int_{t_k}^{t_{k+1}}\nu |\dot \tau(t)|\dd t}
        \leq \int_{t_k}^{t_{k+1}}\nu |\dot \tau(t)|\dd t
    \end{equation}
    with $\dtfric$ from~\cref{ass:friction}. Inserting $|T_k^+|\leq \|A_{t_k}^+ A_{t_k}^{-1}\||T_k|$ into~\eqref{eq:KL_boundproof3} gives
    \begin{equation}
        \begin{aligned}
            \E[|T_{t_{k+1}}|^2]
            \lesssim\exp{-\int_{t_k}^tc'\lambda_{\min}(Q_s) - 2\sqrt{2}\dtfric |\dot \tau(s)|\dd s}
            \E[|T_{t_k}|^2] + \int_{t_k}^{t_{k+1}} \exp{-c'\int_{s}^t\lambda_{\min}(Q_r)\dd r}\E[|b_s|^2]\dd s.
        \end{aligned}
    \end{equation}
    We have
    \begin{equation}
        \begin{aligned}
            c'\lambda_{\min}(Q_s)- 2\sqrt{2}\dtfric |\dot \tau(s)|
            \geq \frac{c'c_0}{T-s} - \frac{2\sqrt{2}\dtfric|\dot\chi(\frac{s}{T})|}{T}
            \geq \frac{c'c_0}{T-s} - \frac{2\sqrt{2}\dtfric\chibound(1-\frac{s}{T})^2}{T}.
        \end{aligned}
    \end{equation}
    Thus, is $T\gtrsim \dtfric$ we can choose $c'$ independent of the problem parameters to achieve that $c'\lambda_{\min}(Q_s) - 2\sqrt{2}\nu |\dot \tau(s)|\geq \frac{c'c_0}{2(T-t)} = \frac{c''c_0}{T-t}$ with $c'' = c'/2$. It follows
    \begin{equation}
        \begin{aligned}
            \E[|T_{t_{k+1}}|^2]
            \lesssim \exp{-c''\int_{t_k}^t\lambda_{\min}(Q_s)\dd s}
            \E[|T_{t_k}|^2] + \int_{t_k}^{t_{k+1}} \exp{-c''\int_{s}^{t_{k+1}}\lambda_{\min}(Q_r)\dd r}\E[|b_s|^2]\dd s\\
        \end{aligned}
    \end{equation}
    and iterating the bound and solving the integrals over $\lambda_{\min}(Q_t)$ concludes the proof upon overwriting the notation of $c'$ by $c''$.
\end{proof}

\begin{proof}[Proof of~\cref{thm:continuous}]
    Let~\eqref{eq:defin_aux_ref} be initialized at the same law as~\eqref{eq:an_under_sde}. Comparing the two via Girsanov (\cf~\cref{sec:girsanov_rigorous}) yields with $T_t$ as in~\eqref{eq:shifted coordinates}
    \begin{equation}\label{eq:proof_backward_KL1}
        \begin{aligned}
            \KL(Z_{T-\epsilon}^\aux{}|Z_{T-\epsilon})
            \leq& \int_0^T \frac{1}{4\friction_{\tau(t)}}\E[|\eta_t^x(\hat X_t - X_t^{\aux{}}) + \eta_t^v(\hat V_t - V_t^{\aux{}})|^2]\dd t\\
            \lesssim& \int_0^T \frac{1}{\friction_{\tau(t)}} \hat{\friction}_t^2(\eta^v_t)^2 \E[|T_t|^2]\dd t
        \end{aligned}
    \end{equation}
    where we insert a small $\epsilon>0$ as the \gls{sde} defining $Z^\aux{}$ degenerates for $t\to T$.
    We have by~\cref{lemma:T_bound_varying_fric}
    \begin{equation}\label{eq:KL_boundproof3}
        \begin{aligned}
            \E[|T_t|^2]
            \lesssim& \bigg(\frac{T-t}{T}\bigg)^{c_0c'}
            \E[|T_0|^2] + \int_0^t \frac{(T-t)^{c'c_0}}{(T-s)^{c_0c'-1}}\E[|b_s|^2]\dd s\\
        \end{aligned}
    \end{equation}
    for some small $c'$ and sufficiently large $c_0$. It follows
    \begin{equation}
        \begin{aligned}
            \|\hat Z_t - Z_t^\aux{}\|^2 
            \lesssim& \hat{\friction}_t^2\|T_t\|^2\\
            \lesssim& \frac{1}{(T-t)^2}\|T_t\|^2\\
            \lesssim& \frac{(T-t)^{c_0c'-2}}{T^{c_0c'}}
            \E[|T_0|^2] + \int_0^t \frac{(T-t)^{c'c_0-2}}{(T-s)^{c_0c'-1}}\E[|b_s|^2]\dd s\\
            =& \frac{(T-t)^{c_0c'-2}}{T^{c_0c'}}\E[|T_0|^2] + \int_0^t \frac{|\dot\tau(s)|(T-t)^{c'c_0-2}}{(T-s)^{c_0c'-1}}|\pi'(\tau(s))|^2|\dot\tau(s)|\dd s\\
            \leq& \frac{(T-t)^{c_0c'-2}}{T^{c_0c'}}\E[|T_0|^2] + \int_0^t \frac{\chibound(T-t)^{c'c_0-2}}{T^3(T-s)^{c_0c'-3}}|\pi'(\tau(s))|^2|\dot\tau(s)|\dd s\\
        \end{aligned}
    \end{equation}
    where we used~\cref{ass:schedule} in the last inequality. Choose $c_0$ such that $c'c_0>2$, then, 
    \begin{equation}
        \begin{aligned}
            \int_0^T \sup_{t\in [0,T]}\frac{\chibound(T-t)^{c'c_0-2}}{T^3(T-s)^{c_0c'-3}}|\pi'(\tau(s))|^2|\dot\tau(s)|\dd s
            \leq& \frac{\chibound}{T^2} \int_0^T |\pi'(\tau(s))|^2|\dot\tau(s)|\dd s\\
            =& \frac{\chibound}{T^2}\int_0^1|\pi'(\tau)|^2\dd \tau\\
            =&\frac{\chibound}{T^2}\Ac(\pi)<\infty.
        \end{aligned}
    \end{equation}
    Therefore, we may use dominated convergence to obtain $Z^\aux{}_t-\hat Z_t\to 0$ as $t\to T$ in $L^2(\P)$.
    Inserted into the $\KL$ bound from~\eqref{eq:proof_backward_KL1} and using lower semi-continuity of $\KL$, we obtain
    \begin{equation}\label{eq:KL_boundproof2}
        \begin{aligned}
            \KL(\hat Z_T|Z_T) \leq \liminf_{\epsilon\to 0}\KL(Z_{T-\epsilon}^\aux{}|Z_{T-\epsilon})
            \lesssim& \int_0^T \frac{1}{\friction_{\tau(t)}}\hat{\friction}_t^2(\eta^v_t)^2 
            \bigg(\frac{T-t}{T}\bigg)^{c_0c'}
            \dd t\; \E[|T_0|^2] \\
            &+ \int_0^T \frac{1}{\friction_{\tau(t)}}\hat{\friction}_t^2(\eta^v_t)^2\int_0^t \frac{(T-t)^{c'c_0}}{(T-s)^{c_0c'-1}}\E[|b_s|^2]\dd s
            \dd t\\
            =& (I_1+I_2)
        \end{aligned}
    \end{equation}
    We begin by computing second integral. Choosing $c_0$ such that $c'c_0\geq 4$, we obtain
    \begin{equation}
        \begin{aligned}
        I_2 \lesssim& \int_0^T \frac{1}{\friction_{\tau(t)}}(\friction_{\tau(t)}^2 + (\eta_t^v)^2)(\eta_t^v)^2\int_0^t \frac{(T-t)^{c'c_0}}{(T-s)^{c_0c'-1}}\E[|b_s|^2]\\
        \lesssim& \int_0^T \frac{\friction_{\tau(t)}}{(T-t)^2} \int_0^t \frac{(T-t)^{c'c_0}}{(T-s)^{c'c_0-1}}
        \E[|b_s|^2]\dd s \dd t + \int_0^T \frac{1}{(T-t)^4} \int_0^t \frac{(T-t)^{c'c_0}}{(T-s)^{c'c_0-1}}
        \E[|b_s|^2] \dd s \dd t\\
        =& \friction_\infty\int_0^T \frac{\E[|b_s|^2]}{(T-s)^{c'c_0-1}} \int_s^T (T-t)^{c'c_0-2} \dd t\dd s + \int_0^T \frac{\E[|b_s|^2]}{(T-s)^{c'c_0-1}} \int_s^T (T-t)^{c'c_0-4} \dd t\dd s\\
        \lesssim& \friction_\infty\int_0^T \E[|b_s|^2]\dd s + \int_0^T \frac{\E[|b_s|^2]}{(T-s)^{2}} \dd s\\
        =& \friction_\infty\int_0^T |\pi'(\tau(s)|^2 |\dot\tau(s)|^2\dd s + \int_0^T \frac{|\pi'(\tau(s)|^2 |\dot\tau(s)|^2}{(T-s)^{2}} \dd s\\
        \overset{(*)}{\lesssim}& \frac{\friction_\infty}{T}\int_0^T |\pi'(\tau(s)|^2 |\dot\tau(s)|\dd s + \frac{1}{T^3}\int_0^T |\pi'(\tau(s)|^2 |\dot\tau(s)| \dd s\\
        \overset{(**)}{=}& \big(\frac{\friction_\infty}{T} + \frac{1}{T^3}\big)\Ac(\pi)
    \end{aligned}
\end{equation}
where we used that by~\cref{ass:schedule}, $\frac{|\dot\tau|}{(T-s)^2}\leq \frac{\chibound}{T^3}$ in $(*)$ and applied a the transformation rule in $(**)$. For the first integral in~\eqref{eq:KL_boundproof2} we compute similarly
\begin{equation}
        \begin{aligned}
            I_1 
            =& \int_0^T \frac{1}{\friction_{\tau(t)}}\hat{\friction}_t^2(\eta^v_t)^2 
            \bigg(\frac{T-t}{T}\bigg)^{c_0c'}
            \dd t\; \E[|T_0|^2]\\
            \lesssim& \int_0^T \frac{\friction_\infty}{(T-t)^2} 
            \bigg(\frac{T-t}{T}\bigg)^{c_0c'}
            \dd t\; \E[|T_0|^2] + \int_0^T \frac{1}{(T-t)^4}
            \bigg(\frac{T-t}{T}\bigg)^{c_0c'}
            \dd t\; \E[|T_0|^2]\\
            \lesssim& \big(\frac{\friction_\infty}{T} + \frac{1}{T^3}\big)\E[|T_0|^2].
        \end{aligned}
    \end{equation}
    The proof is completed upon optimizing over all couplings of $\hat Z_0$ and $Z^\aux{}_0$ to obtain the Wasserstein-2 distance.
\end{proof}
\subsubsection{Proof of~\cref{thm:continuous_strongly}}\label{sec:proof_cont_strongly}
\begin{lemma}\label{lemma:T_bound_varying_fric_strongly}
    If $T\gtrsim \dtfric\omega^{-1}$ then it holds for some absolute constant $c'>0$
    \begin{equation}
        \begin{aligned}
            \E[|T_t|^2]
            \leq& \exp{-c'\omega t} \bigg(\frac{1-\exp{-\omega(T-t)}}{1-\exp{-\omega T}}\bigg)^{c_0c'}
            \E[|T_0|^2] + \int_0^t \exp{-c'\omega (t-s)} \bigg(\frac{1-\exp{-\omega(T-t)}}{1-\exp{-\omega(T-s)}}\bigg)^{c_0c'}\frac{\E[|b_s|^2]}{\lambda_{\min}(Q_s)}\dd s\\
        \end{aligned}
    \end{equation}
\end{lemma}
\begin{proof}
    The proof is similar to~\cref{lemma:T_bound_varying_fric}. However, now $\lambda_{\min}(Q_t)\geq c'(\omega + \frac{c_0\omega}{\exp{\omega(T-t)}-1})$. For small $t$, we cannot dominate $2\sqrt{2}|\dot\tau(t)|\nu$ by $\frac{c_0\omega}{\exp{\omega(T-t)}-1}$ as the former is only $\Oc(T^{-1})$. Instead, the condition $\dtfric T^{-1}\leq \frac{c'\omega}{2}$ ensures again for some $c''<c'$, $c'\lambda_{\min}(Q_t) - 2\sqrt{2}|\dot \tau(s)|\nu\leq c''\lambda_{\min}(Q_t)$. The remaining proof is identical to~\cref{lemma:T_bound_varying_fric}.
\end{proof}
\begin{proof}
    Again we first make sure that $\hat Z_t-Z^\aux{}_t\to 0$ as $t\to T$. Using~\cref{lemma:T_bound_varying_fric_strongly} and that for any $r>0$, $\frac{1-1/r}{r-1} = \frac{1}{r}$
    \begin{equation}
        \begin{aligned}
            \|\hat Z_t - Z_t^\aux{}\|^2 
            \lesssim& \hat{\friction}_t^2 \|T_t\|^2 \\
            \lesssim& \exp{-c'\omega t}\frac{\omega^2}{(\exp{\omega (T-t)}-1)^2} \bigg(\frac{1-\exp{-\omega(T-t)}}{1-\exp{-\omega T}}\bigg)^{c_0c'} \E[|T_0|^2] \\
            &+ \int_0^t \exp{-c'\omega (t-s)} \frac{\omega^2}{(\exp{\omega (T-t)}-1)^2}\bigg(\frac{1-\exp{-\omega(T-t)}}{1-\exp{-\omega(T-s)}}\bigg)^{c_0c'}\frac{\E[|b_s|^2]}{\lambda_{\min}(Q_s)}\dd s\\
            \lesssim& \exp{-c'\omega t}\frac{\omega^2}{\exp{2\omega (T-t)}} \frac{(1-\exp{-\omega(T-t)})^{c'c_0-2}}{(1-\exp{-\omega T})^{c_0c'}} \E[|T_0|^2] \\
            &+ \underbrace{\int_0^t \exp{-c'\omega (t-s)} \frac{\omega^2}{\exp{2\omega (T-t)}}\frac{(1-\exp{-\omega(T-t)})^{c'c_0-2}}{(1-\exp{-\omega(T-s)})^{c'c_0}}\frac{\E[|b_s|^2]}{\lambda_{\min}(Q_s)}\dd s}_{I}
        \end{aligned}
    \end{equation}
    The first term is easily seen to tend to zero as $t\to T$. We insert $\lambda_{\min}(Q_s)$ and obtain for the second term
        \begin{align*}
            I
            \lesssim&\int_0^t \frac{\exp{-c'\omega (t-s)}}{\exp{2\omega (T-t)}}\frac{(1-\exp{-\omega(T-t)})^{c'c_0-2}}{(1-\exp{-\omega(T-s)})^{c'c_0}}\frac{\E[|b_s|^2]}{\lambda_{\min}(Q_s)}\dd s\\
            \lesssim&\int_0^t \frac{\exp{-c'\omega (t-s)}}{\exp{2\omega (T-t)}}\frac{(1-\exp{-\omega(T-t)})^{c'c_0-2}}{(1-\exp{-\omega(T-s)})^{c'c_0}}\frac{(\exp{\omega(T-s)}-1)\E[|b_s|^2]}{\omega\exp{\omega(T-s)}}\dd s\\
            \lesssim&\frac{1}{\omega}\int_0^t \frac{\exp{-c'\omega (t-s)}}{\exp{2\omega (T-t)}}\frac{(1-\exp{-\omega(T-t)})^{c'c_0-2}}{(1-\exp{-\omega(T-s)})^{c'c_0-1}}\E[|b_s|^2]\dd s\\
            \lesssim&\frac{1}{\omega}\int_0^t \frac{\exp{-c'\omega (t-s)}}{\exp{2\omega (T-t)}}\frac{(1-\exp{-\omega(T-t)})^{c'c_0-2}}{(1-\exp{-\omega(T-s)})^{c'c_0-1}}|\dot\tau(s)||\pi'(\tau(s))|^2|\dot\tau(s)|\dd s\\
            \intertext{using that by~\cref{ass:schedule} $|\dot\tau(s)|\lesssim \frac{(T-s)^2}{T^3}$}
            \lesssim&\frac{1}{\omega}\int_0^t \frac{\exp{-c'\omega (t-s)}}{\exp{2\omega (T-t)}}\frac{(1-\exp{-\omega(T-t)})^{c'c_0-2}}{(1-\exp{-\omega(T-s)})^{c'c_0-3}}\frac{(T-s)^2}{T^3(1-\exp{-\omega(T-s)})^2} |\pi'(\tau(s))|^2|\dot\tau(s)|\dd s\\
            \lesssim& \frac{1}{\omega^3T^3}\int_0^t \frac{\exp{-c'\omega (t-s)}}{\exp{2\omega (T-t)}}\frac{(1-\exp{-\omega(T-t)})^{c'c_0-2}}{(1-\exp{-\omega(T-s)})^{c'c_0-3}}\frac{\omega^2(T-s)^2}{(1-\exp{-\omega(T-s)})^2} |\pi'(\tau(s))|^2|\dot\tau(s)|\dd s\\
            \intertext{since $(0,\infty)\ni x\mapsto\big(\frac{x}{1-\exp{- x}}\big)^2$ is increasing}
            \lesssim& \frac{1}{\omega^2T^3}\int_0^t \frac{\exp{-c'\omega (t-s)}}{\exp{2\omega (T-t)}}\frac{(1-\exp{-\omega(T-t)})^{c'c_0-2}}{(1-\exp{-\omega(T-s)})^{c'c_0-3}} \frac{T^2}{1-\exp{-\omega T}}|\pi'(\tau(s))|^2|\dot\tau(s)|\dd s\\
        \end{align*}
    which again tends to zero by dominated convergence as $t\to T$ if $c'c_0>2$. Therefore, as previously by lower semicontinuity $\KL(\hat Z_T|Z_T) = \KL(Z_T^\aux{}|Z_T)$.
    Inserted into the $\KL$ bound from~\eqref{eq:proof_backward_KL1} we obtain
    \begin{equation}
        \begin{aligned}
            \KL(Z_T^\aux{}|Z_T)
            \lesssim& \int_0^T\frac{1}{\friction_{\tau(t)}} \hat{\friction}_t^2(\eta^v_t)^2 
            \exp{-c'\omega t} \bigg(\frac{1-\exp{-\omega(T-t)}}{1-\exp{-\omega T}}\bigg)^{c_0c'}
            \dd t\; \E[|T_0|^2] \\
            &+ \int_0^T\frac{1}{\friction_{\tau(t)}}\hat{\friction}_t^2(\eta^v_t)^2 \int_0^t \exp{-c'\omega (t-s)} \bigg(\frac{1-\exp{-\omega(T-t)}}{1-\exp{-\omega(T-s)}}\bigg)^{c_0c'}\frac{\E[|b_s|^2]}{\lambda_{\min}(Q_s)}\dd s
            \dd t\\
            =& I_1+I_2.
        \end{aligned}
    \end{equation}
    We estimate $I_2$ as
    \begin{align*}
        I_2 \lesssim& \int_0^T\frac{1}{\friction_{\tau(t)}} (\friction_{\tau(t)}^2 + (\eta_t^v)^2)(\eta_t^v)^2\int_0^t \exp{-c'\omega(t-s)}\left(\frac{1 - \exp{-\omega(T-t)}}{1 - \exp{-\omega(T-s)}}\right)^{c'c_0} \frac{(\exp{\omega(T-s)}-1)\E[|b_s|^2]}{\omega \exp{\omega(T-s)}} \dd s \dd t \\
        \lesssim& \int_0^T \frac{\friction_{\tau(t)}\omega^2}{(\exp{\omega(T-t)}-1)^2} \int_0^t \exp{-c'\omega(t-s)}\left(\frac{1 - \exp{-\omega(T-t)}}{1 - \exp{-\omega(T-s)}}\right)^{c'c_0} \frac{(\exp{\omega(T-s)}-1)\E[|b_s|^2]}{\omega \exp{\omega(T-s)}} \dd s \dd t\\
        &+ \int_0^T \frac{\omega^4}{(\exp{\omega(T-t)}-1)^4} \int_0^t \exp{-c'\omega(t-s)}\left(\frac{1 - \exp{-\omega(T-t)}}{1 - \exp{-\omega(T-s)}}\right)^{c'c_0} \frac{(\exp{\omega(T-s)}-1)\E[|b_s|^2]}{\omega\exp{\omega(T-s)}} \dd s \dd t\\
        \intertext{using again that for any $r>0$, $\frac{1-1/r}{r-1} = \frac{1}{r}$, respectively, $\frac{r-1}{1-1/r} = r$}
        \lesssim& \friction_{\tau(t)}\omega \int_0^T \exp{-2\omega(T-t)}\left(1 - \exp{-\omega(T-t)}\right)^{c'c_0-2} \int_0^t \exp{-c'\omega(t-s)}\frac{\exp{\omega(T-s)}}{\left(1 - \exp{-\omega(T-s)}\right)^{c'c_0-1}} \frac{\E[|b_s|^2]}{\exp{\omega(T-s)}} \dd s \dd t\\
        &+ \omega^3\int_0^T \exp{-4\omega(T-t)} \left(1 - \exp{-\omega(T-t)}\right)^{c'c_0-4}\int_0^t \exp{-c'\omega(t-s)}\frac{\exp{\omega(T-s)}}{\left(1 - \exp{-\omega(T-s)}\right)^{c'c_0-1}} \frac{\E[|b_s|^2]}{ \exp{\omega(T-s)}} \dd s \dd t\\
        \lesssim& \friction_{\tau(t)}\omega\int_0^T \frac{\E[|b_s|^2]}{\left(1 - \exp{-\omega(T-s)}\right)^{c'c_0-1}} 
        \int_s^T \exp{-2\omega(T-t)}\left(1 - \exp{-\omega(T-t)}\right)^{c'c_0-2}\exp{-c'\omega(t-s)} \dd t\dd s \\
        &+ \omega^3\int_0^T \frac{\E[|b_s|^2]}{\left(1 - \exp{-\omega(T-s)}\right)^{c'c_0-1}} 
        \int_s^T \exp{-4\omega(T-t)} \left(1 - \exp{-\omega(T-t)}\right)^{c'c_0-4} \exp{-c'\omega(t-s)} \dd t \dd s\\
        \intertext{choosing $c_0$, such that $c'c_0 > 4$ and noting that $c'<2$}
        \lesssim& \friction_{\tau(t)}\int_0^T \frac{\E[|b_s|^2]}{\left(1 - \exp{-\omega(T-s)}\right)^{2}} 
        \exp{-c'\omega (T-s)}\frac{1 - \exp{-(2-c')\omega(T-s)}}{1 - \exp{-\omega(T-s)}}\dd s \\
        &+ \omega^2\int_0^T \frac{\E[|b_s|^2]}{\left(1 - \exp{-\omega(T-s)}\right)^{2}} 
        \exp{-c'\omega (T-s)}\frac{1 - \exp{-(4-c')\omega(T-s)}}{1 - \exp{-\omega(T-s)}}\dd s\\
        \intertext{using that for any $\alpha>0$, $\sup_{x\in (0,\infty)}\frac{1-\exp{-\alpha x}}{1-\exp{-x}}<\infty$}
        \lesssim& (\friction_{\tau(t)}+ \omega^2)\int_0^T \frac{|\pi'(\tau(s))|^2|\dot\tau(s)|^2}{\left(1 - \exp{-\omega(T-s)}\right)^{2}}
        \exp{-c'\omega (T-s)}\dd s \\
        \intertext{using that by~\cref{ass:schedule} $|\dot \tau(s)|\leq \chibound\frac{(T-s)^2}{T^3}$}
        \lesssim& \frac{(\friction_{\tau(t)}+\omega^2)}{T^3}\int_0^T |\pi'(\tau(s))|^2|\dot\tau(s)|\left(\frac{T-s}{1 - \exp{-\omega(T-s)}}\right)^{2}
        \exp{-c'\omega (T-s)}\dd s \\
        =& \frac{(\friction_{\tau(t)}+\omega^2)}{\omega^2 T^3}\int_0^T |\pi'(\tau(s))|^2|\dot\tau(s)|\left(\frac{\omega(T-s)}{1 - \exp{-\omega(T-s)}}\right)^{2}
        \exp{-c'\omega (T-s)}\dd s\\
        \intertext{and using that $(0,\infty)\ni x\mapsto\big(\frac{x}{1-\exp{- x}}\big)^2\exp{-c' x}$ is bounded}
        \lesssim& \frac{(\friction_{\tau(t)}+\omega^2)}{\omega^2 T^3}\int_0^T |\pi'(\tau(s))|^2|\dot\tau(s)\dd s \\
        =& \frac{(\friction_{\tau(t)}+\omega^2)}{\omega^2 T^3}\Ac(\pi)
\end{align*}
Similarly, we find for the first integral
\begin{equation}
        \begin{aligned}
            I_1 
            =& \int_0^T \frac{1}{\friction_{\tau(t)}}\hat{\friction}_t^2(\eta^v_t)^2 
            \exp{-c'\omega t} \bigg(\frac{1-\exp{-\omega(T-t)}}{1-\exp{-\omega T}}\bigg)^{4}
            \dd t\; \E[|T_0|^2]\\
            \lesssim& \int_0^T \frac{\friction_{\tau(t)}\omega^2}{(\exp{\omega(T-t)}-1)^2} 
            \exp{-c'\omega t} \bigg(\frac{1-\exp{-\omega(T-t)}}{1-\exp{-\omega T}}\bigg)^{4}
            \dd t\; \E[|T_0|^2]\\
            &+ \int_0^T \frac{\omega^4}{(\exp{\omega(T-t)}-1)^4}
            \exp{-c'\omega t} \bigg(\frac{1-\exp{-\omega(T-t)}}{1-\exp{-\omega T}}\bigg)^{4}
            \dd t\; \E[|T_0|^2]\\
            =& \int_0^T \friction_{\tau(t)}\omega^2 \exp{-2\omega(T-t)}
            \exp{-c'\omega t} \frac{(1-\exp{-\omega(T-t)})^{2}}{(1-\exp{-\omega T})^{4}}
            \dd t\; \E[|T_0|^2]\\
            &+ \int_0^T \omega^4\exp{-4\omega(T-t)}
            \exp{-c'\omega t} \frac{1}{(1-\exp{-\omega T})^{4}}
            \dd t\; \E[|T_0|^2]\\ 
            \leq& \frac{\friction_{\tau(t)}\omega^2 + \omega^4}{(1-\exp{-\omega T})^{4}}\int_0^T  \exp{-2\omega(T-t)}
            \exp{-c'\omega t}
            \dd t\; \E[|T_0|^2]\\
            =& \frac{\friction_{\tau(t)}\omega^2 + \omega^4}{(1-\exp{-\omega T})^{4}}\exp{-c'\omega T} \int_0^T  \exp{-(2-c')\omega(T-t)}
            \dd t\; \E[|T_0|^2]\\
            \leq& \frac{\friction_{\tau(t)}\omega + \omega^3}{(1-\exp{-\omega T})^{4}}\exp{-c'\omega T} \; \E[|T_0|^2]
        \end{aligned}
    \end{equation}
    concluding the proof again by optimizing over all couplings of $(\hat Z_0,Z_0)$.
\end{proof}

\subsection{Discretization}
\subsubsection{Explicit update for the stochastic exponential Euler discretization}\label{sec:ee_discretization_explicit}
The explicit solution of~\cref{eq:discretization} is (\cf~\citep[Lemma 5.3.9]{Chewi26Book})
    \begin{equation}\label{eq:EE_explicit}
        \begin{cases}
            \discX_t =& \discX_{t_k} + \friction_k^{-1}(1-\exp{-\friction_k(t-t_k)})\discV_{t_k}
            \\
            &-\left[\friction_k^{-1}(t-t_k)-\friction_k^{-2}(1-\exp{-\friction_k(t-t_k)}) \right]\nabla_x \pot_{\tau(t_k)}(\discX_{t_k})\\
            & + \sqrt{2\friction_k}\int_{t_k}^t\int_{t_k}^s \exp{-\friction_k(s-\sigma)}\dd W^v_\sigma \dd s \\
            \discV_t =& \exp{-\friction_k(t-t_k)}\discV_{t_k} - \frac{1}{\friction_k}(1-\exp{-\friction_k(t-t_k)})\nabla_x \pot_{\tau(t_k)}(\discX_{t_k}) + \sqrt{2\friction_k}\int_{t_k}^t \exp{-\friction_k(t-s)}\dd W^v_s.
        \end{cases}
    \end{equation}
    In particular, we have
    \begin{equation}\label{eq:EE_update}
        \begin{bmatrix}
            \discX_t\\
            \discV_t
        \end{bmatrix}\bigg|\begin{bmatrix}
            \discX_{t_k}\\
            \discV_{t_k}
        \end{bmatrix}\sim
        \Nc\left(\begin{bmatrix}
            \mean^x_t(\discX_{t_k},\discV_{t_k})\\
            \mean^v_t(\discX_{t_k},\discV_{t_k})
        \end{bmatrix},
        \begin{bmatrix}
            \var^{x,x}_t(\discX_{t_k},\discV_{t_k}) & \var^{x,v}_t(\discX_{t_k},\discV_{t_k})\\
            \var^{x,v}_t(\discX_{t_k},\discV_{t_k}) & \var^{v,v}_t(\discX_{t_k},\discV_{t_k})
        \end{bmatrix}\right)
    \end{equation}
    where the conditional mean is
    \begin{align}\label{eq:EE_specifics}
        \mean^x_t(x,v) &= x + \friction_k^{-1}(1-\exp{-\friction_k(t-t_k)})v \notag\\
        &-\left[\friction_k^{-1}(t-t_k)-\friction_k^{-2}(1-\exp{-\friction_k(t-t_k)}) \right]\nabla_x \pot_{\tau(t_k)}(x),\notag \\
        \mean^v_t(x,v) &= \exp{-\friction_k(t-t_k)}v - \frac{1}{\friction_k}(1-\exp{-\friction_k(t-t_k)})\nabla_x \pot_{\tau(t_k)}(x),\notag\\
        \intertext{and the conditional covariance,}
        \Sigma^{x,x}_t(x,v) &= \frac{2}{\friction_k}\left[t-t_k +  \frac{1}{2\friction_k}\left(1-\exp{-2\friction_k(t-t_k)}\right) - \frac{2}{\friction_k}\left(1-\exp{-\friction_k(t-t_k)}\right)\right]I_d,\\
        \Sigma^{v,v}_t(x,v) &= \left(1-\exp{-2\friction_k(t-t_k)}\right) I_d,\notag\\
        \Sigma^{x,v}_t(x,v) &= \frac{2}{\friction_k}\left(1-\exp{-\friction_k(t-t_k)}\right) I_d - \frac{1}{\friction_k}\left(1-\exp{-2\friction_k(t-t_k)}\right) I_d\notag.
    \end{align}

\subsubsection{Proof of~\cref{lemma:disc_error}}\label{sec:proof_disc}
\begin{proof}
    Applying Girsanov's theorem to compare~\eqref{eq:an_under_sde} and~\eqref{eq:discretization} we obtain
    \begin{align*}
            \KL(Z_t|\discZ_T)
            \lesssim& \sum_{k=0}^{N-1}\frac{1}{\friction_k}\int_{t_k}^{t_{k+1}}\E[|\nabla\pot_{\tau(t)}(X_t) - \nabla\pot_{\tau(t_k)}(X_{t_k})|^2]\dd t\\
            \lesssim& \sum_{k=0}^{N-1}\frac{1}{\friction_k}\int_{t_k}^{t_{k+1}}\E[|\nabla\pot_{\tau(t)}(X_t) - \nabla\pot_{\tau(t_k)}(X_{t_k})|^2] \\
            &\hspace{1.5cm}+ \E[|s_{\theta}(\tau(t_k),X_{t_k}) + \nabla\pot_{\tau(t_k)}(X_{t_k})|^2]\dd t\\
            \intertext{using~\cref{ass} and~\cref{lemma:underdamped_score_bound} with $M_x = \sup_{t>0}\E[|X_t|^2]$ and $M_v = \sup_{t>0}\E[|V_t|^2]$}
            \lesssim& \sum_{k=0}^{N-1}\frac{1}{\friction_k}\int_{t_k}^{t_{k+1}}(t-t_k)^2\left\{\frac{\dtbound_{\tau(t_{k+1})}^2}{T^2} (1+M_x) + \lip_{[\tau(t_k),\tau(t_{k+1})]}^2M_v\right\}\dd t\\
            &\hspace{1.5cm}+\step_k \E[|s_{\theta}(\tau(t_k),X_{t_k}) + \nabla\pot_{\tau(t_k)}(X_{t_k})|^2]
            \\
            \lesssim& \sum_{k=0}^{N-1}\step_k^3 \frac{1}{\friction_k}\left\{\frac{\dtbound_{\tau(t_{k+1})}^2}{T^2} (1+M_x) + \lip_{\tau(t_{k+1})}^2M_v\right\}  + \frac{\epsilon_{\score}^2}{\friction_0}
            \intertext{using that $T>\frac{\minimizerbound\sqrt{\lip_\infty}}{\sqrt{d m_\infty}}$ together with~\cref{lemma:moment_bound_moving}}
            \lesssim& \sum_{k=0}^{N-1}\step_k^3 \frac{1}{\friction_k}\left\{\frac{\dtbound_{\tau(t_{k+1})}^2}{T^2} (1+\frac{d}{m_{\tau(t_{k+1})}}) +\frac{\lip_{\tau(t_{k+1})}^3 d}{m_{\tau(t_{k+1})}}\right\} + \frac{\epsilon_{\score}^2}{\friction_0}
    \end{align*}
\end{proof}

\section{Girsanov's theorem and its implications}\label{ssec:girsanov_general}

\begin{theorem}[Girsanov, \citep{karatzas2014brownian}, Theorem 5.1, Chapter 3]\label{theorem:girsaonv}
    Let $T>0$ and $(X_t)_{t\geq 0}$ be an adapted process with
    \begin{equation}
        \P\left [\int_0^T|X_s|^2\dd s<\infty \right ] = 1
    \end{equation}
    and $(W_t)_t$ a Brownian motion with respect to the measure $\P$. Define the process
    \begin{equation}
        M_t = \exp{\int_0^t X_s\cdot \dd W_s - \frac{1}{2}\int_0^t|X_s|^2\dd s}.
    \end{equation}
    If $Z_t$ is a martingale, then the process $\tilde W_t$ defined as
    \begin{equation}
        \tilde W_t = W_t - \int_0^t X_s\dd s, \quad t\in[0,T]
    \end{equation}
    is a Brownian motion with respect to the measure $\Q$ defined via $\frac{\dd Q}{\dd P} = M_T$.
\end{theorem}
The main condition for Girsaonv to be applicable is $(Z_t)_t$ being a martingale. The following condition, referred to as \emph{Novikov's} condition, is sufficient.
\begin{theorem}[Novikov's condition, \citep{karatzas2014brownian}, Theorem 5.13, Chapter 3]
    In the setting of~\cref{theorem:girsaonv}, if there exists $\delta>0$ such that for every $t\geq 0$
    \begin{equation}
        \P\left[\exp{\frac{1}{2}\int_t^{t+\delta}|X_s|^2\dd s}<\infty\right]=1
    \end{equation}
    then $M$ is a martingale and Girsanov's theorem applies.
\end{theorem}

\begin{coroll}\label{coroll:girs_KL}
    Let $(X_t)_t,(Y_t)_t$ two random processes in $\R^n$ defined via 
    \begin{equation}
        \dd X_t = b(t,X_t)\dd t + \sigma\dd W_t,\quad \dd Y_t = \{b(t,Y_t) - U_t\}\dd t + \sigma\dd W_t.
    \end{equation}
    for some random $U_t$ and $\sigma\in \R^{n\times m}$. Assume there exists $\theta_t$ such that $U_t = \sigma \theta$ and 
    \begin{equation}
        M_t \coloneqq \exp{\int_0^t \theta_s\cdot \dd W_s - \frac{1}{2}\int_0^t |\theta_s|^2\dd s}
    \end{equation}
    is a martingale, it holds true that 
    \begin{equation}
        \KL(X_T|Y_T)\leq \KL((X_t)_{t\in [0,T]})|(Y_t)_{t\in [0,T]})) \leq \KL(X_0|Y_0) + \int_0^T \frac{1}{2}\E[|\theta_t|^2]\dd t.
    \end{equation}
\end{coroll}
Note that Novikov's condition is often impractical to verify. Luckily,~\cref{coroll:girs_KL} can usually be applied nonetheless relying on localization. We omit the details and refer to~\citep[Section 1.1.1]{Chewi26Book}, \citep{chen2022sampling}.

\section{A Transport equation for $\pi(\tau(t))$}
\begin{theorem}[Transport equation for AC curves, Theorems 8.3.1, 8.4.5 in~\citep{ambrosio2005gradient}]\label{thm:optimal_velocity}
    Let $\mu:[0,T]\rightarrow\Pc_2(\R^d)$ be absolutely continuous and $|\mu'|$ its metric derivative. Then there exists a velocity field $b(t,x)$ with $b(t)\in L^2(\mu(t))$ for a.e. $t$ such that $\mu$ satisfies the transport equation
    \begin{equation}\label{eq:optimal velocity}
        \partial_t \mu = \nabla\cdot(-b\mu).
    \end{equation}
    Moreover, all velocity fields such that~\cref{eq:optimal velocity} holds true satisfy $\|b\|_{L^2(\mu(t)}\geq |\mu'(t)|$ for a.e. $t$ and there exists a unique velocity field $b$ such that $\|b\|_{L^2(\mu(t)} = |\mu'(t)|$ with \(b = \nabla \varphi\) with \(\varphi \in \text{Cyl}(\R^d)\) for a.e. $t$.
\end{theorem}

\begin{lemma}\label{lemma:optimal_velocity_independent}
    Let $\mu_x(t),\mu_v(t)$, $t\in[0,T]$ be absolutely continuous curves in $\Pc_2(\R^{d})$. Let $\hat b_x(t)\in L^2(\mu_x(t))$, $\hat b_v(t)\in L^2(\mu_v(t))$ be the uniquely determined smallest vector fields generating $\mu_x(t),\mu_v(v)$ according to~\cref{thm:optimal_velocity}. 
    Then it holds true that $(\hat b_x(t),\hat b_v(t))\in L^2(\mu_x(t)\otimes\mu_v(t))$ is the smallest velocity field generating the path $\mu = \mu_x\otimes\mu_v$.
\end{lemma}
\begin{proof}
    First of all, it is clear that $(\hat b_x(t),\hat b_v(t))\in L^2(\mu_x(t)\otimes\mu_v(t))$ for a.e. $t$ as
    \begin{equation}
        \iint |(\hat b_x(t,x),\hat b_v(t,v))|^2\dd \mu_x(t,x)\dd \mu_v(t,v)
        = \|b_x(t)\|_{L^2(\mu_x(t))}^2+\|b_v(t)\|_{L^2(\mu_v(t))}^2.
    \end{equation}
    Moreover, one can easily see that $(\hat b_x(t),\hat b_v(t))$ in fact generates the path $\mu$. 
    In order to prove that $(\hat b_x(t),\hat b_v(t))$ is the minimal velocity field, let $(b_x(t,x,v),b_v(t,x,v))$ be an arbitrary velocity field generating $\mu$. We will show that its $L^2$ norm is a.s. lower-bounded by that of $\hat b$.
    Note that the continuity equation with $(b_x,b_v)$ reads as
    \begin{equation}
        \begin{aligned}
            \partial_t \mu
            = (\partial_t\mu_x)\mu_v + \mu_x(\partial_t\mu_v)
            =&\nabla\cdot \begin{bmatrix}
                \begin{pmatrix}
                    b_x\\
                    b_v
                \end{pmatrix}
                \mu
            \end{bmatrix}\\
            =& \mu_v \nabla_x\cdot (b_x \mu_x) + \mu_x \nabla_v\cdot (b_v\mu_v)\\
        \end{aligned}
    \end{equation}
    Integrating over $v$ yields
    \begin{equation}
        \begin{aligned}
            \partial_t\mu_x
            =& \nabla_x\cdot \bigg(\int b_x(t,x,v)\dd \mu_v(t,v) \mu_x\bigg).
        \end{aligned}
    \end{equation}
    and similarly for integrating over $x$. That is, the $x$ and $v$ marginals satisfy transport equations with the $v$- and $x$-averaged velocity fields, $\bar b_x(t,x) := \int b_x(t,x,v)\dd \mu_v(t,v)$ and $\bar b_v(t,v):= \int b_v(t,x,v)\dd \mu_x(t,x)$, respectively. But then, by minimality of $\hat b_x$ and $\hat b_v$ it follows for a.e. $t$
    \begin{equation}
        \begin{aligned}
            \|\hat b_x(t)\|_{L^2(\mu_x(t))}&\leq \|\bar b_x(t)\|_{L^2(\mu_x(t))},\text{ and}\\
            \|\hat b_v(t)\|_{L^2(\mu_v(t))}&\leq \|\bar b_v(t)\|_{L^2(\mu_v(t))}
        \end{aligned}
    \end{equation}
    and, thus, by Hölder's inequality for a.e. $t$ we have
    \begin{equation}
        \begin{aligned}
            \|(\hat b_x(t),&\hat b_v(t))\|_{L^2(\mu(t))}^2 
            \leq \int |\bar b_x(t)|^2\dd \mu_x + \int |\bar b_v(t)|^2\dd \mu_v\\
            =& \int \bigg| \int b_x(t,x,v)\dd\mu_v(t,v)\bigg|^2\dd \mu_x(t,x) + \int \bigg| \int b_v(t,x,v)\dd\mu_x(t,x)\bigg|^2\dd \mu_v(t,v)\\
            \leq& \iint | b_x(t,x,v)|^2\dd\mu_v(t,v)\dd \mu_x(t,x) + \iint |b_v(t,x,v)|^2\dd\mu_x(t,x)\dd \mu_v(t,v)\\
            =& \|(b_x(t),b_v(t))\|_{L^2(\mu(t))}^2
        \end{aligned}
    \end{equation}
    concluding the proof.
\end{proof}
\section{Moment bounds}
We have the following standard result bounding the minimizer of the potential,\cf~\citep{durmus2019high}.
\begin{lemma}\label{lemma:bound_minimizer}
    Denote $x^*_\tau = \arg\min_x \pot_\tau(x)$ and assume $|\int x \pi(\tau,\dd x)|=\Oc(\sqrt{d/m_\tau})$ for all $\tau$. Then, it holds $|x^*_\tau|\lesssim \sqrt{d/m_\tau}$.
\end{lemma}
\begin{proof}
    Let $(X^\tau_t)_t$ follow
    \begin{equation}
        \dd X^\tau_t = -\nabla \pot_\tau (X^\tau_t)\dd t + \sqrt{2} \dd W_t
    \end{equation}
    that is, the Langevin diffusion targeting $\pi(\tau)$ for \emph{frozen} $\tau$. By \ito{}'s lemma
    \begin{equation}
        \begin{aligned}
            \dd |X^\tau_t - x^*_\tau|^2 
            &= \{ -2\inner{X^\tau_t - x^*_\tau}{\nabla \pot_\tau (X^\tau_t) - \nabla \pot_\tau (x^*_\tau)} + d\}\dd t + M_t\\
            &\leq \{ -2m_\tau|X^\tau_t - x^*_\tau|^2 + d\}\dd t + M_t,\\
        \end{aligned}
    \end{equation}
    where \(M_t\) is a martingale.
    Taking the expectation and applying Grönwall's lemma yields
    \begin{equation}
        \E[|X^\tau_t - x^*_\tau|^2]
        \leq \exp{-2m_\tau t}\left (\E[|X^\tau_0 - x^*_\tau|^2] - \frac{d}{2m_\tau}\right) + \frac{d}{2m_\tau}.
    \end{equation}
    Taking the limit $t\rightarrow 0$ yields
    \begin{equation}
        \lim_{t\rightarrow \infty}\E[|X^\tau_t - x^*_\tau|^2] = \int |x-x^*_\tau|^2 \pi(\tau,\dd x) \leq \frac{d}{m_\tau}.
    \end{equation}
    Lastly, denoting $\bar x_\tau = \int x \pi(\tau,\dd x)$ we have
    \begin{equation}
        |\bar x_\tau - x^*_\tau|^2 
        = |\E_{\pi(\tau)}[X-x^*_\tau]|^2 
        \leq \E_{\pi(\tau)}[|X-x^*_\tau|^2]
        \leq \frac{d}{m_\tau}
    \end{equation}
    and consequently, $|x^*_\tau|\leq |\bar x_\tau| + |x^*_\tau - \bar x_\tau| \lesssim \sqrt{d/m_\tau}$.
\end{proof}

\begin{lemma}[Moment bound for moving reference]\label{lemma:moment_bound_moving}
    Let $\friction = 2\sqrt{\lip_\infty}$. Define 
    \begin{equation}
        \begin{aligned}
            \omega_\tau(x,v) 
            \coloneqq& 2|x-x_\tau^*|^2 + \frac{4}{\friction}\inner{x-x_\tau^*}{v} + \frac{4}{\friction^2}|v|^2\\
        \end{aligned}
    \end{equation}
    where $x_\tau^*\in\argmin_x \pot_\tau(x)$. Denote the generator at time $t$ of~\eqref{eq:an_under_sde} as $\Lc_{\tau(t)}$. Then it holds true that 
    \begin{equation}
        \Lc_{\tau(t)} \omega_{\tau(t)}(x,v)
            \leq -\frac{m_\infty}{8\sqrt{\lip_\infty}}\omega_{\tau(t)}(x,v)+\frac{4d}{\sqrt{\lip_\infty}} + 
            \frac{\minimizerbound^2}{T^2}\bigg(\frac{4(m_\infty + \lip_\infty)}{m_\infty \sqrt{\lip_\infty}} + \frac{2}{\sqrt{\lip_\infty}} \bigg)
    \end{equation}
    and consequently, 
    \begin{equation}
        \begin{aligned}
            \E[\omega_{\tau(t_k)}(X_{t_k},V_{t_k})]
            \lesssim \E[\omega_{\tau(0)}(X_{0},V_{0})] 
            + \frac{4d}{m_{\tau_{k-1}}} + 
            \frac{8 \minimizerbound^2\lip_{\tau_{k-1}}}{T^2 m_{\tau_{k-1}}^2 }.
        \end{aligned}
    \end{equation}
    In particular, if $\E[\omega_{\tau(0)}(X_0,V_0)]\lesssim \frac{d}{m_0} + \frac{\minimizerbound^2\lip_0}{T^2m_0^2}$, than $\E[\omega_{\tau(t)}(X_t,V_t)]\lesssim \frac{d}{m_{\tau(t)}} + \frac{\minimizerbound^2\lip_{\tau(t)}}{T^2m_{\tau(t)}^2}$ for all $t\geq 0$.
\end{lemma}
\begin{proof}
    In the following we drop the time dependence from the notation and write $\tau$ instead of $\tau(t)$ for simplicity.
    By \ito{}'s lemma within $(\tau_k,\tau_{k+1})$ where $\friction_\tau=\friction_k$ we find
    \begin{equation}\label{eq:lyapunov_cont1}
        \begin{aligned}
            \Lc_\tau \omega_\tau(x,v)
            \leq& -\frac{4}{\friction_\tau}\inner{x-x^*_\tau}{\nabla\pot_\tau(x)}
            -\frac{4}{\friction_\tau}|v|^2
            -\frac{8}{\friction_\tau^2}\inner{v}{\nabla\pot_\tau(x)}
            +\frac{8d}{\friction_\tau} \\
            &+ 4\inner{x^*_\tau - x + \frac{1}{\friction_\tau} v}{\dot \tau  \partial_\tau x^*_{\tau(t)}}
        \end{aligned}
\end{equation}
    Using~\citep[Theorem 2.1.12]{nesterov2013introductory} we find
    \begin{equation}
        \begin{aligned}
            \inner{x-x^*_\tau}{\nabla \pot_\tau(x)}
            \geq& \frac{m_\tau \lip_\tau}{m_\tau + \lip_\tau}|x-x^*_\tau|^2 + \frac{1}{m_\tau + \lip_\tau}|\nabla \pot_\tau(x)|^2
        \end{aligned}
    \end{equation}
    Moreover, for any $\delta>0$ which (will be chosen in a moment) again by Young
    \begin{equation}
        \inner{v}{\nabla\pot_\tau(x)}
        \leq \frac{1}{2\delta}|v|^2 + \frac{\delta}{2}|\nabla\pot_\tau(x)|^2.
    \end{equation}
    Inserting these estimates in~\eqref{eq:lyapunov_cont1} we obtain
    \begin{equation}
        \begin{aligned}
            \Lc_\tau \omega_\tau(x,v)
            \leq& -\frac{4}{\friction_\tau}\frac{m_\tau \lip_\tau}{m_\tau + \lip_\tau}|x-x^*_\tau|^2 -|v|^2\big(\frac{4}{\friction_\tau} - \frac{4}{\delta\friction_\tau^2}\big)\\
            &+|\nabla\pot_\tau(x)|^2 \{\frac{4\delta}{\friction_\tau^2} - \frac{4}{\friction_\tau(m_\tau + \lip_\tau)}\}
            +\frac{8d}{\friction_\tau} + 4\inner{x^*_\tau - x + \frac{1}{\friction_\tau} v}{\dot x^*_\tau}.
        \end{aligned}
    \end{equation}
    Now we pick $\delta = \frac{\friction_\tau}{m_\tau+\lip_\tau}$ so that the term involving $\nabla\pot_\tau$ vanishes.
    Moreover, in order for the coefficient of $|v|^2$ to be negative, we require
    \begin{equation}
        1<\delta\friction_\tau = \frac{\friction_\tau^2}{m_\tau + \lip_\tau}
    \end{equation}
    which is true for the choice $\friction_\tau = 2\sqrt{\lip_\tau}$. Using that by~\cref{ass,ass:schedule} $|\dot \tau \partial_\tau x^*_\tau|\lesssim \frac{\minimizerbound}{T}$ and applying Young we find
    \begin{equation}\label{eq:lyapunov_cont2}
        \begin{aligned}
            \Lc_\tau \omega_\tau(x,v)
            \leq& -\frac{2}{\friction_\tau}\frac{m_\tau \lip_\tau}{m_\tau + \lip_\tau}|x-x^*_\tau|^2 -\frac{1}{\friction_\tau}|v|^2
            +\frac{8d}{\friction_\tau} + 
            \frac{\minimizerbound^2}{T^2}\big(\frac{2\friction_\tau(m_\tau + \lip_\tau)}{m_\tau \lip_\tau} + \frac{4}{\friction_\tau} \big)
        \end{aligned}
    \end{equation}
    Note that we can write
    \begin{equation}
        \omega_\tau(x,v) 
        = \begin{bmatrix}
            x-x^*_\tau\\
            v
        \end{bmatrix}^\top
        \begin{bmatrix}
            2& \frac{2}{\friction_\tau}\\
            *& \frac{4}{\friction_\tau^2}
        \end{bmatrix}
        \begin{bmatrix}
            x-x^*_\tau\\
            v
        \end{bmatrix}.
    \end{equation}
    Thus, in view of~\eqref{eq:lyapunov_cont2}, to ensure a drift condition of the form $\Lc_\tau \omega_\tau\leq -\lambda \omega_\tau + b$ with $\lambda,b>0$ we have to ensure that
    \begin{equation}
        \begin{bmatrix}
            -\frac{2}{\friction_\tau}\frac{m_\tau \lip_\tau}{m_\tau + \lip_\tau} + 2\lambda& \lambda\frac{2}{\friction_\tau}\\
            *& -\frac{1}{\friction_\tau}+ \lambda \frac{4}{\friction_\tau^2}
        \end{bmatrix}\preceq 0
    \end{equation}
    which is the case if the above matrix admits negative trace and positive determinant, which is true for $\lambda = \frac{m_\tau}{8\sqrt{\lip_\tau}}$.
    Indeed, with this choice the trace reads as
    \begin{equation}
        \begin{aligned}
            -\frac{2}{\friction_\tau}\frac{m_\tau \lip_\tau}{m_\tau + \lip_\tau} + 2\lambda -\frac{1}{\friction_\tau}+ \lambda \frac{4}{\friction_\tau^2}
            =&-\frac{m_\tau \sqrt{\lip_\tau}}{m_\tau + \lip_\tau} + \frac{m_\tau}{4\sqrt{\lip_\tau}} -\frac{1}{2\sqrt{\lip_\tau}}+ \frac{m_\tau}{8 \lip_\tau^{3/2}}\\
            \leq&-\frac{m_\tau}{\sqrt{\lip_\tau}} + \frac{m_\tau}{4\sqrt{\lip_\tau}} -\frac{1}{2\sqrt{\lip_\tau}}+ \frac{1}{8 \sqrt{\lip_\tau}}\\
            <&0
        \end{aligned}
    \end{equation}
    where we used that $m_\tau\leq \lip_\tau$, respectively $m_\tau/\lip_\tau\leq 1$. For the determinant, on the other hand, we obtain
    \begin{equation}
        \begin{aligned}
            (-\frac{2}{\friction_\tau}&\frac{m_\tau \lip_\tau}{m_\tau + \lip_\tau} + 2\lambda)(-\frac{1}{\friction_\tau}+ \lambda \frac{4}{\friction_\tau^2}) -  \lambda^2\frac{4}{\friction_\tau^2}\\
            =& \frac{2m_\tau\lip_\tau}{\friction_\tau^2(m_\tau+\lip_\tau)} - \frac{8\lambda m_\tau\lip_\tau}{\friction_\tau^3(m_\tau+\lip_\tau)} - \frac{2\lambda}{\friction_\tau} + \frac{8\lambda^2}{\friction_\tau^2} - \frac{4\lambda^2}{\friction_\tau^2}\\
            =& \frac{m_\tau}{2(m_\tau+\lip_\tau)} - \frac{ m_\tau^2}{8\lip_\tau(m_\tau+\lip_\tau)} - \frac{m_\tau}{8\lip_\tau} + \frac{m_\tau^2}{32\lip_\tau^2} - \frac{m_\tau^2}{64\lip_\tau^2}\\
            \geq& \frac{m_\tau}{2(m_\tau+\lip_\tau)} - \frac{ m_\tau^2}{8\lip_\tau(m_\tau+\lip_\tau)} - \frac{m_\tau}{8\lip_\tau} + \frac{m_\tau^2}{64\lip_\tau^2}\\
            \geq& \frac{m_\tau}{\lip_\tau}\big(\frac{1}{4} - \frac{1}{16} - \frac{1}{8}\big) + \frac{m_\tau^2}{128\lip_\tau^2}>0.
        \end{aligned}
    \end{equation}
    It follows,
    \begin{equation}\label{eq:moment_bound_omega}
        \begin{aligned}
            \Lc_\tau \omega_\tau(x,v)
            \leq& -\frac{m_\tau}{8\sqrt{\lip_\tau}}\omega_\tau(x,v)+\frac{8d}{\friction_\tau} + \frac{4}{\friction_\tau}(m_\tau + 
            \frac{\minimizerbound^2}{T^2}\big(\frac{2\friction_\tau(m_\tau + \lip_\tau)}{m_\tau \lip_\tau} + \frac{4}{\friction_\tau} \big)\\
            \lesssim& -\frac{m_\tau}{8\sqrt{\lip_\tau}}\omega_\tau(x,v) + \frac{4d}{\sqrt{\lip_\tau}} + 
            \frac{\minimizerbound^2}{T^2}\big(\frac{4(m_\tau + \lip_\tau)}{m_\tau \sqrt{\lip_\tau}} + \frac{2}{\sqrt{\lip_\tau}} \big)
        \end{aligned}
    \end{equation}
    Now we note that with the notation 
    \begin{equation}
        A=\begin{bmatrix}
            1&0\\
            1&\frac{2}{\gamma_k}
        \end{bmatrix},\quad 
        A_+=\begin{bmatrix}
            1&0\\
            1&\frac{2}{\gamma_{k+1}}
        \end{bmatrix}
    \end{equation}
    we have that 
    \begin{equation}
        \omega_{\tau_k}(x,v) 
        = |A_+(x,v)|
        = |A_+A^{-1}A(x,v)|
        \leq \|A_+A^{-1}\|\omega_{\tau_{k+1}}(x,v)
    \end{equation}
    and it holds $\|A_+A^{-1}\|\leq 1+\sqrt{2}(1-r_k)$ with $r_k = \frac{\friction_k}{\friction_{k+1}}$. As in the proof of~\cref{lemma:T_bound_varying_fric} we have $\|A_+A^{-1}\|\leq \exp{\sqrt{2}\int_{t_k}^{t_{k+1}} \nu |\dot \tau(s)|\dd s}\leq \exp{\sqrt{2}\int_{t_k}^{t_{k+1}} \frac{\nu\chibound}{T} \dd s}$ with $\chibound$ from \cref{ass:schedule}. Then, integrating~\eqref{eq:moment_bound_omega} over $(t_k,t_{k+1})$ yields
    \begin{equation}
        \begin{aligned}
            \E[\omega_{\tau(t_{k+1})}(X_{t_{k+1}},V_{t_{k+1}})]
            \lesssim& \exp{-\frac{m_\tau}{8\sqrt{\lip_\tau}}\step_k}\exp{\sqrt{2}\int_{t_k}^{t_{k+1}} \frac{\nu\chibound}{T} \dd s}\E[\omega_{\tau(t_k)}(X_{t_{k+1}},V_{t_{k+1}})] 
            \\
            &+ \big\{\frac{4d}{\sqrt{\lip_\tau}} + 
            \frac{\minimizerbound^2}{T^2}\big(\frac{4(m_\tau + \lip_\tau)}{m_\tau \sqrt{\lip_\tau}} + \frac{2}{\sqrt{\lip_\tau}} \big)\big\}\frac{1-\exp{-\frac{m_\tau}{\sqrt{\lip_\tau}}\step_k}}{\frac{m_\tau}{\sqrt{\lip_\tau}}}\\
            \leq& \exp{-\step_k(\frac{m_\tau}{8\sqrt{\lip_\tau}}-\sqrt{2}\frac{\nu\chibound}{T})}\E[\omega_{\tau(t_k)}(X_{t_{k+1}},V_{t_{k+1}})] 
            \\
            &+ \big\{\frac{4d}{\sqrt{\lip_\tau}} + 
            \frac{\minimizerbound^2}{T^2}\big(\frac{4(m_\tau + \lip_\tau)}{m_\tau \sqrt{\lip_\tau}} + \frac{2}{\sqrt{\lip_\tau}} \big)\big\}\step_k
        \end{aligned}
    \end{equation}
    and by choosing $T\geq \frac{\sqrt{2\lip_\tau}\chibound\dtfric}{16 m_\tau}$
    \begin{equation}
        \begin{aligned}
            \E[\omega_{\tau(t_{k+1})}(X_{t_{k+1}},V_{t_{k+1}})]
            \lesssim& \exp{-\step_k\frac{m_\tau}{16\sqrt{\lip_\tau}}}\E[\omega_{\tau(t_k)}(X_{t_{k+1}},V_{t_{k+1}})] 
            \\
            &+ \big\{\frac{4d}{\sqrt{\lip_\tau}} + 
            \frac{8 \minimizerbound^2\sqrt{\lip_\tau}}{T^2 m_\tau }\big\}\step_k
        \end{aligned}
    \end{equation}
    Iterating this bound yields
    \begin{equation}
        \begin{aligned}
            \E[\omega_{\tau(t_k)}(X_{t_k},V_{t_k})]
            \lesssim& \exp{-\sum_{\ell=0}^{k-1}\step_\ell\frac{m_{\tau_\ell}}{16\sqrt{\lip_{\tau_\ell}}}}\E[\omega_{\tau(0)}(X_{0},V_{0})] 
            \\
            &+ \sum_{\ell=0}^{k-1}\exp{-\sum_{n=\ell-1}^{k-2}\step_\ell\frac{m_{\tau_\ell}}{16\sqrt{\lip_{\tau_\ell}}}}\big\{\frac{4d}{\sqrt{\lip_{\tau_\ell}}} + 
            \frac{8 \minimizerbound^2\sqrt{\lip_{\tau_\ell}}}{T^2 m_{\tau_\ell} }\big\}\step_\ell
        \end{aligned}
    \end{equation}
    In particular, noting that $\lip_{\ell}\leq \lip_k$ and $m_\ell\geq m_k$ we have in the constant step size setting
    \begin{equation}
        \begin{aligned}
            \E[\omega_{\tau(t_k)}(X_{t_k},V_{t_k})]
            \lesssim& \exp{-\frac{k\step m_{\tau_{k-1}}}{16\sqrt{\lip_{\tau_{k-1}}}}\sum_{\ell=0}^{k-1}\step_\ell}\E[\omega_{\tau(0)}(X_{0},V_{0})] 
            \\
            &+ \sum_{\ell=0}^{k-1}\exp{-(k-\ell-1)\step\frac{m_{\tau_{k-1}}}{16\sqrt{\lip_{\tau_{k-1}}}}}\big\{\frac{4d}{\sqrt{\lip_{\tau_{k-1}}}} + 
            \frac{8 \minimizerbound^2\sqrt{\lip_{\tau_{k-1}}}}{T^2 m_{\tau_{k-1}} }\big\}\step\\
            \lesssim& \exp{-\frac{k\step m_{\tau_{k-1}}}{16\sqrt{\lip_{\tau_{k-1}}}}\sum_{\ell=0}^{k-1}\step_\ell}\E[\omega_{\tau(0)}(X_{0},V_{0})] 
            + \frac{\{\frac{4d}{\sqrt{\lip_{\tau_{k-1}}}} + 
            \frac{8 \minimizerbound^2\sqrt{\lip_{\tau_{k-1}}}}{T^2 m_{\tau_{k-1}} }\big\}}{\frac{m_{\tau_{k-1}}}{16\sqrt{\lip_{\tau_{k-1}}}}}\\
            \lesssim& \E[\omega_{\tau(0)}(X_{0},V_{0})] 
            + \frac{4d}{m_{\tau_{k-1}}} + 
            \frac{8 \minimizerbound^2\lip_{\tau_{k-1}}}{T^2 m_{\tau_{k-1}}^2 }.
        \end{aligned}
    \end{equation}
\end{proof}

\begin{lemma}[Underdamped moment bound]\label{lemma:underdamped_score_bound}
    It holds for any $0<s<t$
    \begin{equation}
        \E[|\nabla \pot_{\tau(t)}(X_t) - \nabla \pot_{\tau(s)}(X_s) |^2] \lesssim (s-t)^2\left\{\frac{\dtbound_{\tau(t)}^2\chibound^2}{T^2}(1+\sup_{r\in [s,t]}\E[|X_r|^2]) + \lip_{\tau(t)}^2\sup_{r\in [s,t]}\E[|V_r|^2]\right\}
    \end{equation}
\end{lemma}
\begin{proof}
    By \ito{}'s lemma it holds
    \begin{equation}
        \dd \nabla \pot_{\tau(t)}(X_t) = \{\dot \tau(t) \partial_\tau \nabla \pot_{\tau(t)}(X_t) + \nabla^2 \pot_{\tau(t)}(X_t)V_t\}\dd t.
    \end{equation}
    Thus, we obtain by~\cref{ass},\cref{ass:schedule}
    \begin{equation}
        \begin{aligned}
            \E[|\nabla \pot_{\tau(t)}(X_t) - \nabla \pot_{\tau(s)}(X_s) |^2] 
            =& \E[|\int_s^t \dot \tau(r)\partial_\tau \nabla \pot_{\tau(r)}(X_r) + \nabla^2\pot_{\tau(r)}(X_r)V_r\dd r|^2]\\
            \lesssim& (s-t)^2 \frac{\chibound^2\dtbound_{\tau(t)}^2}{T^2} (1+\sup_{r\in [s,t]}\E[|X_r|^2]) \\
            &+ \lip_{\tau(t)}^2(s-t)^2\sup_{r\in [s,t]}\E[|V_r|^2]
        \end{aligned}
    \end{equation}
\end{proof}

\section{Proof of~\cref{example:conv}: The convolutional path satisfies all assumptions}\label{sec:convolutional}
\begin{proof}
We consider only the first schedule variant. \Cref{ass:schedule},~\cref{ass:schedule1} is trivial to verify. To show~\cref{ass:schedule},~\cref{ass:schedule2} we compute
\begin{equation}
    \begin{aligned}
        \frac{|\dot\chi(t)|}{(1-t)^2} = \frac{(1+\cos(\pi t))\pi\sin(\pi t)}{(1-t)^2} = \frac{1+\cos(\pi t)}{1-t}\frac{\pi\sin(\pi t)}{1-t}
    \end{aligned}
\end{equation}
where both terms admit a finite limit for $t\to 1$ by L'H\^opital. Note that $\pi_\tau = \frac{1}{\tau^{d/2}(1-\tau)^{d/2}}\pi_0(\frac{\cdot}{\sqrt{1-\tau}})*\gamma(\frac{\cdot}{\sqrt{\tau}})$ with $\gamma$ the standard Gaussian distribution. Thus, $p(\tau)$ is smooth for all $\tau$ by properties of the convolution. 
Then, by~\cref{lemma:tau_dep_hessian_bounds} we have that $p(\tau)$ is $\lip_\tau$-smooth and $m_\tau$ strongly logconcave with $\lip_\tau = \frac{\lip_0}{1+(\lip_0-1)\tau}$, $m_\tau = \frac{m_0}{1+(m_0-1)\tau}$ so that~\cref{ass}, \cref{ass:pot1} is satisfied.
The density $p(\tau)$ satisfies the Fokker-Planck equation
\begin{equation}
    \partial_\tau p = \nabla\cdot(\frac{x}{2(1-\tau)}p) + \frac{1}{2(1-\tau)}\Delta p
\end{equation}
from which it follows that $\log p$ satisfies
\begin{equation}
    \begin{aligned}
        \partial_\tau \log p = \frac{1}{2(1-\tau)}\left(d+x\cdot\nabla\log p + \Delta \log p + |\nabla\log p|^2\right)
    \end{aligned}
\end{equation}
and, therefore, $\nabla \log p = -\nabla \pot_\tau$ satisfies
\begin{equation}\label{eq:dtaupot}
    \begin{aligned}
        \partial_\tau \nabla \pot_\tau = \frac{1}{2(1-\tau)}\left(\nabla\pot_\tau + \nabla^2\pot_\tau x + \nabla \Delta \pot_\tau + 2\nabla^2\pot_\tau \nabla\pot_\tau \right).
    \end{aligned}
\end{equation}
Thus, using also that $\tau\leq 1-\kappa$, we find
\begin{equation}
    \begin{aligned}
        |\partial_\tau \nabla \pot_\tau(x)| \leq \frac{1}{2\kappa}\left(|\nabla\pot_\tau(x)|+ \lip_\tau |x| + |\nabla \Delta \pot_\tau| + 2\lip_\tau| \nabla\pot_\tau(x)|\right).
    \end{aligned}
\end{equation}
By Lipschitzness we have $|\nabla\pot_\tau(x)|\leq \lip_\tau |x-x^*_\tau|$ and we show in~\cref{lemma:cumulant_bound} that under the assumed bound on the third derivatives, $|\nabla^3\pot_\tau(x)|\leq \frac{M(1-\tau)^{3/2}}{4(m_0\tau + (1-\tau))^3}$, consequently, 
\begin{equation}
    \begin{aligned}
        |\partial_\tau\nabla \pot_\tau(x)| \leq \frac{1}{2\kappa}\left(\lip_\tau (1+2\lip_\tau)|x-x^*_\tau|+ \lip_\tau |x| + \frac{d M(1-\tau)^{3/2}}{4(m_0\tau + (1-\tau))^3}\right)
    \end{aligned}
\end{equation}
showing~\cref{ass},\cref{ass:pot2}. 
Lastly, we show~\cref{ass},\cref{ass:pot3}.
Denote for simplicity $\partial_\tau x^*_\tau = \cot x^*_\tau$.
The minimizer $x^*_\tau$ satisfies $\nabla \pot_\tau(x_\tau^*)=0$ for all $\tau$.  Differentiating with respect to $\tau$ yields
\begin{equation}\label{eq:condition_x*tau}
    0 = (\partial_\tau \nabla \pot_\tau)(x_\tau^*) + \nabla^2 \pot_\tau (x^*_\tau) \dot x^*_\tau.
\end{equation}
Inserting $x^*_\tau$ in~\eqref{eq:dtaupot} and noting that by definition $\nabla\pot_\tau(x^*_\tau)=0$ gives
\begin{equation}
    (\partial_\tau \nabla \pot_\tau)(x^*_\tau) = \frac{1}{2(1-\tau)}\left(\nabla^2\pot_\tau(x^*_\tau) x^*_\tau + \nabla \Delta \pot_\tau(x^*_\tau)\right).
\end{equation}
Combining with~\eqref{eq:condition_x*tau} yields
\begin{equation}\label{eq:x_tauderivative}
    \begin{aligned}
        \dot x^*_\tau 
        =& \frac{1}{2(1-\tau)}[\nabla^2 \pot_\tau (x^*_\tau)]^{-1}\left(- \nabla^2\pot_\tau(x^*_\tau) x^*_\tau - \nabla \Delta \pot_\tau(x^*_\tau)\right)\\
        =& \frac{1}{2(1-\tau)}\left(-x^*_\tau - \underbrace{[\nabla^2 \pot_\tau (x^*_\tau)]^{-1}\nabla \Delta \pot_\tau(x^*_\tau)}_{\eqqcolon \psi(\tau)}\right).
    \end{aligned}
\end{equation}
Then it holds
\begin{equation}
    \frac{\dd}{\dd \tau}\bigg( \exp{\int_0^\tau\frac{1}{2(1-\sigma)}\dd \sigma } x_\tau^*\bigg) = -\exp{\int_0^\tau\frac{1}{2(1-\sigma)}\dd \sigma } \frac{\psi(\tau)}{2(1-\tau)}.
\end{equation}
Integrating yields
\begin{equation}
    x_\tau^* = \exp{-\int_0^\tau\frac{1}{2(1-\sigma)}\dd \sigma } x_0^* - \int_0^\tau \exp{-\int_\rho^\tau\frac{1}{2(1-\sigma)}\dd \sigma } \frac{\psi(\rho)}{2(1-\rho)}\dd \rho.
\end{equation}
We have $\exp{-\int_\rho^\tau\frac{1}{2(1-\sigma)}\dd \sigma } = \big(\frac{1-\tau}{1-\rho}\big)^{1/2}$ and therefore
\begin{equation}\label{eq:x_taubound}
    x_\tau^* = (1-\tau)^{1/2} x_0^* - \int_0^\tau \big(\frac{1-\tau}{1-\rho}\big)^{1/2} \frac{\psi(\rho)}{2(1-\rho)}\dd \rho.
\end{equation}
Using that by~\cref{lemma:tau_dep_hessian_bounds,lemma:cumulant_bound} it holds 
\begin{equation}
    |\psi(\tau)| \leq \frac{d M(1-\tau)^{3/2}}{4m_\tau(m_0\tau + (1-\tau))^3}
\end{equation}
it finally follows from~\eqref{eq:x_taubound}
\begin{equation}\label{eq:x_taubound}
    \begin{aligned}
        |x_\tau^*| 
        \leq& (1-\tau)^{1/2} |x_0^*| + (1-\tau)^{1/2} \int_0^\tau  \frac{dM(1-\rho)^{3/2}}{8m_\rho(m_0\rho + (1-\rho))^3}\dd
        \rho\\
        \leq& (1-\tau)^{1/2} \bigg(|x_0^*| +  \frac{dM}{8m_\rho(m_0\wedge 1)^3}\bigg)\\
        \leq& (1-\tau)^{1/2} \bigg(\sqrt{\frac{d}{m_0}} +  \frac{dM}{8m_\rho(m_0\wedge 1)^3}\bigg)
    \end{aligned}
\end{equation}
where we used~\cref{lemma:bound_minimizer} to bound $|x^*_0|$.
Inserting into~\eqref{eq:x_tauderivative} and using again that $\tau<1-\kappa$ concludes the proof.
\end{proof}

\section{Helper Results}

\subsection{On the upper and lower bound of \(\nabla^2 \log \pot_\tau\)}

\begin{lemma}\label{lemma:tau_dep_hessian_bounds}
    Let $\pot_0$ be $m$-strongly convex and $L$-smooth and define $\pi(\tau) = \law(\sqrt{1-\tau}X + \sqrt{\tau}Z)$ where $X\sim \pi(0)$, $Z\sim \Nc(0,\id)$, $X\perp Z$. Then it holds for $p(\tau,x) = \frac{\dd \pi(\tau)}{\dd x}(x)$
    \begin{equation}
        \frac{m}{1 + (m-1)\tau}\id \preceq -\nabla^2 \log p(\tau,x) \preceq \frac{L}{1 + (L-1)\tau}\id.
    \end{equation}
\end{lemma}
\begin{proof}
    By \citep[Lemma 5]{mikulincer2023lipschitzheatflow} we have
    \begin{equation}
        -\nabla^2 \log \pi_\tau \preceq \frac{L}{1 + (L-1)\tau}\id
    \end{equation}
    and by~\citep{saumard2014log}
    \begin{equation}
        -\nabla^2 \log \pi_\tau \succeq \frac{m}{1 + (m-1)\tau}\id.
    \end{equation}
\end{proof}

\subsection{Third derivative bound for the convolutional path}
\begin{lemma}\label{lemma:cumulant_bound}
    Assume there exists $M>0$ such that the operator norm of the third derivative satisfies $\|\nabla^3 \log p(0, x)\|_{op} \leq M$ for all $x$. And define for $\tau\in (0,1)$, $\pot_\tau(x) = -\log p(\tau,x)$ with 
    \begin{equation}
        p(\tau,x) = \frac{p\big(0,\frac{\cdot}{\sqrt{1-\tau}}\big)}{(1-\tau)^{d/2}}*\frac{\gamma \big(\frac{\cdot}{\sqrt{\tau}}\big)}{\tau^{d/2}}
    \end{equation}
    with $\gamma$ the standard Gaussian distribution.
    Then it holds true that
    \begin{equation}
        \|\nabla^3 \pot_{\tau}\|_{op}\leq \frac{M(1-\tau)^{3/2}}{4(m_0\tau + (1-\tau))^3}.
    \end{equation}
\end{lemma}
\begin{proof}
    Let $u,v,w\in\R^d$ normalized vectors. One can easily verify that 
    \begin{equation}\label{eq:third_cumulant}
        \nabla^3 \pot_\tau(z)[u,v,w] = \frac{1}{\tau^3}\int f_uf_vf_w \dd \mu_z(x)
    \end{equation}
    where we define the tilted distribution
    \begin{equation}
        \mu_z(\dd x) \propto \exp{-\pot_0(\frac{x}{\sqrt{1-\tau}}) - \frac{|x-z|^2}{2\tau}}
    \end{equation}
    and $f_h(x) = (x-\E_{\mu_z}[X])\cdot h$ for $h\in \{u,v,w\}$.
    Denote moreover $U_z(x) = \pot_0(\frac{x}{\sqrt{1-\tau}}) + \frac{|x-z|^2}{2\tau}$ and the differential operator $\Lc \coloneqq -\nabla U_z\cdot + \Delta$. By strong log-concavity of $\mu_z$ and Lax-Milgram, there exists a unique, smooth, zero-mean $\phi_u$ such that $-\Lc\phi_u = f_u$ and by a standard bootstrapping argument $\phi_u\in \Cc^\infty$.
    We will in the following make frequent use of the fact that for any $f,g$, $\int -\Lc(f) g\dd \mu_z = \int \nabla f\cdot \nabla g\dd \mu_z$.
    We find
    \begin{equation}
        \begin{aligned}
            \int f_uf_vf_w \dd \mu_z(x) 
            = \int (-\Lc\phi_u)f_vf_w \dd \mu_z(x)
            =& \int \nabla \phi_u\cdot \nabla(f_vf_w) \dd \mu_z(x)\\
            =& \int \nabla \phi_u\cdot (w f_v + v f_w) \dd \mu_z(x).
        \end{aligned}
    \end{equation}
    Repeating this trick with $-\Lc\phi_v = f_v$, $-\Lc\phi_w = f_w$ we obtain
    \begin{equation}\label{eq:cum_est1}
        \begin{aligned}
            \int f_uf_vf_w \dd \mu_z(x) 
            =& \int \nabla \phi_v\cdot \nabla^2 \phi_u w + \nabla \phi_w\cdot \nabla^2 \phi_u v\, \dd \mu_z(x).
        \end{aligned}
    \end{equation}
    We proceed by bounding the $L^2(\mu_z)$-norms of $\nabla \phi_h$, $h\in\{v,w\}$, and $\nabla^2 \phi_u$. 
    We begin with $\nabla \phi_h$. 
    By definition of $\phi_h$ we have for $\psi_h\coloneqq \nabla\phi_h$
    \begin{equation}\label{eq:psi_pde}
        -\Lc(\psi_h) = h - \nabla^2 U_z\psi_h
    \end{equation}
    where we use the same notation $\Lc$ for the corresponding differential operator now acting on vector-valued functions.
    By~\eqref{eq:psi_pde} and $(\frac{m_0}{1-\tau}+\frac{1}{\tau})$-strong convexity of $U_z$
    \begin{equation}
        \begin{aligned}
            0\leq \int |\nabla\psi|^2\dd \mu_z
            = \int -\Lc(\psi_h)\psi_h\dd \mu_z
            =& \int h\cdot \psi_h - \psi_h^\top \nabla^2 U\psi_h\dd \mu_z\\
            \leq& \int h\cdot \psi_h - (\frac{m_0}{1-\tau}+\frac{1}{\tau})|\psi_h|^2\dd \mu_z.
        \end{aligned}
    \end{equation}
    Since $h = \nabla f_h$ and $\psi_h = \nabla \phi_h$ we, therefore, have
    \begin{equation}
        \begin{aligned}
            (\frac{m_0}{1-\tau}+\frac{1}{\tau})\int|\psi_h|^2\dd \mu_z\leq \int h\cdot \psi_h\dd \mu_z = \int \nabla f_h \cdot \nabla \phi_h \dd \mu_z 
            =& \int f_h(-\Lc \phi_h)\dd \mu_z \\
            =& \int |f_h|^2\dd \mu_z\\
            \overset{(*)}{\leq}& \frac{1}{\frac{m_0}{1-\tau}+\frac{1}{\tau}}\int |\nabla f_h|^2\dd \mu_z\\
            \overset{(**)}{=}& \frac{\tau(1-\tau)}{\tau m_0+(1-\tau)}
        \end{aligned}
    \end{equation}
    where we used the Poincar\'e inequality for strongly log-concave densities in $(*)$ and $|\nabla f_h| = |h|=1$ in $(**)$.
    Next we bound $\nabla^2\phi_h$. Differentiating~\eqref{eq:psi_pde} once more yields for $\Phi_h\coloneqq \nabla^2 \phi_h$ the equation
    \begin{equation}\label{eq:LPhi}
        -\Lc(\Phi_h)= -\nabla^3 U_z [\nabla \phi_h,\cdot,\cdot] -\nabla^2U_z\Phi_h - \Phi_h\nabla^2 U_z.
    \end{equation}
    where $\Lc$ denotes the appropriate action on matrix-valued functions and $\nabla^3 U_z$ is a 3d tensor. We have again by integration by parts
    \begin{equation}
        \int -\Lc(\Phi)\cdot \Phi\dd \mu_z = \int \|\nabla \Phi\|_F^2\dd \mu_z\geq 0.
    \end{equation}
    Combining with~\eqref{eq:LPhi}, $\nabla^2 U_z \succeq \frac{m_0}{1-\tau}+\frac{1}{\tau}$ and 
    \begin{equation}
        \begin{aligned}
            2(\frac{m_0}{1-\tau}+\frac{1}{\tau}) \int \|\Phi\|_F^2\dd \mu 
            \leq& \int (\nabla^2U_z\Phi + \Phi\nabla^2 U_z)\cdot\Phi\dd \mu_z \\
            \leq& \int -\nabla^3 U_z [\nabla \phi,\cdot,\cdot]\cdot \Phi\dd \mu_z\\
            \leq& \frac{M}{(1-\tau)^{3/2}} \bigg(\int |\nabla\phi|^2\dd\mu_z\bigg)^{1/2}\bigg(\int \|\Phi\|_F^2\dd\mu_z\bigg)^{1/2}
        \end{aligned}
    \end{equation}
    Thus,
    \begin{equation}
        \begin{aligned}
            \int \|\Phi\|_F^2\dd \mu 
            \leq& \big(\frac{\tau(1-\tau)M}{2(1-\tau)^{3/2}(m_0\tau + (1-\tau))}\big)^2 \int |\nabla\phi|^2\dd\mu_z\\
            \leq& \big(\frac{\tau(1-\tau)M}{2(1-\tau)^{3/2}(m_0\tau + (1-\tau))}\big)^2\big(\frac{\tau(1-\tau)}{m_0\tau + (1-\tau)}\big)^2
        \end{aligned}
    \end{equation}
    Inserting into~\eqref{eq:cum_est1} then leads to
    \begin{equation}
        \begin{aligned}
            \big|\int f_uf_vf_w \dd \mu_z(x) \big|
            \leq& \int |\nabla \phi_v| |\nabla^2 \phi_u| + |\nabla \phi_w| |\nabla^2 \phi_u|\, \dd \mu_z(x)\\
            \leq& \int |\nabla \phi_v| |\nabla^2 \phi_u| + |\nabla \phi_w| |\nabla^2 \phi_u|\, \dd \mu_z(x)\\
            \leq& \frac{M \tau^3(1-\tau)^3}{(1-\tau)^{3/2}(m_0\tau + (1-\tau))^3}\\
            \leq& \frac{M \tau^3(1-\tau)^{3/2}}{(m_0\tau + (1-\tau))^3}.
        \end{aligned}
    \end{equation}
    Inserting into~\eqref{eq:third_cumulant} then yields
    \begin{equation}
        |\nabla^3 \pot_\tau(z)[u,v,w]| \leq \frac{M(1-\tau)^{3/2}}{(m_0\tau + (1-\tau))^3}.
    \end{equation}
\end{proof}

\section{Details for complexity evaluation}\label{sec:complexity}
The target is $\pi_x(0)=\Nc(x^\star,\Lambda^{-1})$ in $d=2$ with $x^\star=(1,1)^\top$ and
$\Lambda=\operatorname{diag}(1000,1)$. We use the variance preserving path
$\pi_x(\tau)=\law(\sqrt{1-\tau}\,X+\sqrt{\tau}\,Z)=\Nc\big(\sqrt{1-\tau}\,x^\star,(1-\tau)\Lambda^{-1}+\tau I_d\big)$
with the schedule $\tau(t)=\chi(t/T)$, $\chi(s)=\big(\tfrac{1+\cos(\pi s)}{2}\big)^2$. We have that $m_\tau\equiv1$, while
$\lip_\tau$ increases from $1$ at $\tau=1$ to $\lip_0=1000$ at $\tau=0$. All methods start from
$\discX_0\sim\Nc(0,I_d)$, and the underdamped methods additionally from $\discV_0\sim\Nc(0,I_d)$. The
scores are exact ($\epsilon_\score=0$).

For log-spaced accuracies $\epsilon^2\in[10^{-4},10^{-1}]$, we report the smallest number of iterations
$K^*$ with $\KL(\pi_x(0)\,|\,\law(\discX_K))\le\epsilon^2$. The iterates are affine in the state plus
Gaussian noise, so their law is Gaussian. We propagate its mean and covariance exactly and evaluate the KL
divergence in closed form. The horizon $T$ is not prescribed. Instead, it is searched for numerically.

Each method uses the step size of its own discretization bound, with all hidden constants set to one and
local constants at $\tau(t)$. ANULD runs with friction $\friction_k=2\sqrt{\lip_{\tau(t_k)}}$
and minimizes the number of steps under the bound $\sum_k\step_k^3w(t_{k+1})\le\epsilon^2/2$ of
\cref{lemma:disc_error}. This gives $\step_k=\eta\,w(t_k)^{-1/3}$ with
$\eta^2\int_0^Tw^{1/3}\dd t=\epsilon^2/2$ and
\begin{equation}
    w(t)=\frac{1}{2\sqrt{\lip_{\tau(t)}}}\left(\frac{\dtbound_{\tau(t)}^2}{T^2}\Big(1+\frac{d}{m_{\tau(t)}}\Big)+\frac{\lip_{\tau(t)}^3\,d}{m_{\tau(t)}}\right).
\end{equation}
For constant $w$, this is the step size of \cref{coroll:disc_complexity}. DALMC \citep{cordero2025non} uses
$\step_k=\eta/\lip_{\tau(t_k)}$. Here $\eta$ solves
$d\eta(1+\eta)\int_0^T\lip_{\tau(t)}\dd t+\eta^2\big(T(M_2+d)+\Ac(\pi)/T\big)=\epsilon^2/2$, the
discretization bound of \cite[Theorem~3.4]{cordero2025non}, with $M_2=\Ex_{\pi_x(0)}[|\cdot|^2]$. ULD runs on the target ($\tau\equiv0$) with $\friction=2\sqrt{\lip_0}$ and a constant step size. The
solid line uses $\step=\epsilon\sqrt{m_0}/(\lip_0\sqrt d)$ from \citet[Theorem~6]{zhang2023improved}. The
dashed line uses the ANULD rule at $\tau\equiv0$, i.e.\ $\step=\epsilon/\sqrt{2Tw_0}$ with
$w_0=\lip_0^{5/2}d/(2m_0)$.

\end{document}